\documentclass[10pt]{amsart}
\usepackage{latexsym}
\usepackage{amscd, amsfonts, eucal, mathrsfs, amsmath, amssymb, amsthm}
\input xy
\xyoption{all}

\usepackage{pdfsync}
\usepackage{newcent}
\usepackage{comment}

\newcommand{\field}[1]{\mathbf #1}
\newcommand{\mf}[1]{\mathfrak #1}
\newcommand{\mc}[1]{\mathcal #1}
\newcommand{\ms}[1]{\mathscr #1}
\newcommand{\widebar}[1]{\overline{#1}}

\newcommand{\R}{\field R}
\renewcommand{\L}{\field L}
\newcommand{\C}{\field C}
\newcommand{\F}{\field F}
\newcommand{\Z}{\field Z}
\newcommand{\Q}{\field Q}

\newcommand{\simto}{\stackrel{\sim}{\to}}

\newcommand{\shom}{\ms H\!om}
\newcommand{\rshom}{\R\shom}

\DeclareMathOperator{\Spec}{Spec}
\DeclareMathOperator{\spec}{Spec}
\DeclareMathOperator{\Spf}{Spf}
\DeclareMathOperator{\spf}{Spf}

\DeclareMathOperator{\Def}{Def}
\newcommand{\Td}{\operatorname{Td}}
\newcommand{\Mom}{s\ms D}

\newcommand{\sMom}{\operatorname{sD}}

\DeclareMathOperator{\Pic}{Pic}

\DeclareMathOperator{\D}{D}

\DeclareMathOperator{\pr}{pr}

\DeclareMathOperator{\Fil}{Fil}

\DeclareMathOperator{\Ext}{Ext}

\DeclareMathOperator{\NS}{NS}

\newcommand{\m}{\boldsymbol{\mu}}

\newcommand{\G}{\field G} 

\renewcommand{\H}{\operatorname{H}}

\DeclareMathOperator{\chern}{ch}

\renewcommand{\]}{]\!\hspace{0.03em}]}

\DeclareMathOperator*{\tensor}{\otimes}

\DeclareMathOperator*{\ltensor}{\stackrel{\field L}{\otimes}}
\DeclareMathOperator{\Tr}{\operatorname{Tr}}

\newcommand{\inj}{\hookrightarrow}

\newcommand{\id}{\operatorname{id}}

\DeclareMathOperator{\Hom}{\operatorname{Hom}}

\DeclareMathOperator{\M}{\operatorname{M}}

\DeclareMathOperator{\B}{\operatorname{B\!}}

\renewcommand{\mathbb}{\mathbf}

\newcommand{\twist}{\operatorname{tw}}
\newcommand{\Hcris}{H_{\text{\rm cris}}}

\newtheorem{lem}[subsection]{Lemma}

\numberwithin{equation}{subsection}

\newtheorem{thm}[subsection]{Theorem}

\newtheorem{prop}[subsection]{Proposition}

\newtheorem{cor}[subsection]{Corollary}

\theoremstyle{definition}
\newtheorem{defn}[subsection]{Definition}

\newtheorem{notation}[subsection]{Notation}

\theoremstyle{remark}
\newtheorem{remark}[subsection]{Remark}

\newtheorem{pg}[subsection]{}

\newcommand{\mls}[1]{\mathscr{#1}}

\numberwithin{equation}{subsection}
\newcommand{\Sp}{\text{\rm Spec}}

\author{Max Lieblich and Martin Olsson}

\title{Fourier-Mukai partners of \text{\rm K3}  surfaces in positive characteristic}

\begin{document}
\maketitle
\setcounter{tocdepth}{1}
\tableofcontents

\section{Introduction}
\label{sec:introduction}

In this paper we establish several basic facts about Fourier-Mukai equivalence of \text{\rm K3}  surfaces over fields of positive characteristic and develop some foundational material on deformation and lifting of Fourier-Mukai kernels, including the study of several ``realizations'' of Mukai's Hodge structure in standard cohomology theories (\'etale, crystalline, Chow, etc.).

In particular, we prove the following theorem, extending to positive characteristic classical results due to Hosono et al \cite{HLOY}, Mukai \cite{M}, and Orlov \cite{Orlov} in characteristic $0$.  For a scheme  $Z$ of finite type  over a field $k$, let $\D(Z)$ denote the bounded derived category with coherent cohomology.  For a \text{\rm K3}  surface $X$ over an algebraically closed field $k$, we have algebraic moduli spaces $\M_X(v)$ of sheaves with fixed Mukai vector $v$ (see section \ref{modulidef} for the precise definition) that are stable with respect to a suitable polarization.

\begin{thm}\label{T:mainthm}
Let $X$ be a \text{\rm K3}  surface over an algebraically closed field $k$ of positive characteristic $\neq 2$.
\begin{enumerate}
\item If $Y$ is a smooth projective $k$-scheme such that there exists an equivalence of triangulated categories $\D(X)\simeq \D(Y)$, then $Y$ is a \text{\rm K3} surface isomorphic to $\M_X(v)$ for some Mukai vector $v$ such that there exists a Mukai vector $w$ with $\langle v, w\rangle = 1$.
\item There exist only finitely many smooth projective varieties $Y$ over $k$ with $\D(X)\simeq \D(Y)$.
\item If $X$ has Picard number at least 11 and $Y$ is a smooth projective $k$-scheme with $\D(Y)\simeq \D(X)$, then $X\simeq Y$.  In particular, any Shioda-supersingular \text{\rm K3} surface is determined up to isomorphism by its derived category.
\end{enumerate}

\end{thm}
The classical proofs of these results in characteristic $0$ rely heavily on the Torelli theorem and lattice theory, so a transposition into characteristic $p$ is necessarily delicate.  We present here a theory of the ``Mukai motive'', generalizing the Mukai-Hodge structure to other cohomology theories, and use various realizations to aid in lifting derived-equivalence problems to characteristic $0$.

These techniques also yield proofs of several other results.  The
first answers a question of Musta\c{t}\u{a} and Huybrechts, while the
second establishes the truth of the  variational crystalline Hodge
conjecture \cite[Conjecture 9.2]{MP} in some special cases. (In the
course of preparing this manuscript, we learned that Huybrechts
discovered essentially the same proof of Theorem \ref{T:horse}, in
$\ell$-adic form.)

\begin{thm}\label{T:horse}
If $X$ and $Y$ are \text{\rm K3} surfaces over a finite field $\F$ of characteristic $\neq 2$ such that $\D(X)$ is equivalent to $\D(Y)$, then $X$ and $Y$ have the same zeta-function.  In particular, $\# X(\F)=\# Y(\F).$
\end{thm}

\begin{thm}\label{T:chicken}
Suppose $X$ and $Y$ are \text{\rm K3}  surfaces over an algebraically closed field $k$ of positive characteristic $\neq 2$  with Witt vectors $W$, and that $\mc X/W$ and $\mc Y/W$ are lifts, giving rise to a Hodge  filtration on the $F$-isocrystal
$\Hcris^4(X\times Y/K),$
where $K$ denotes the field of fractions of $W$.
Suppose $Z\subset X\times Y$ is a correspondence coming from a Fourier-Mukai kernel. If the fundamental class of $Z$ lies in $\Fil^2\Hcris^4(X\times Y/K)$ then $Z$ is the specialization of a cycle on $\mc X\times_W\mc Y$.
\end{thm}

Throughout this paper we consider only fields of characteristic $\neq 2$.

\subsection{Outline of the paper}

Sections 2 and 3 contain foundational background material on
Fourier-Mukai equivalences.  In section 2 we discuss variants in other
cohomology theories (\'etale, crystalline, Chow) of Mukai's original
construction of a Hodge structure associated to a smooth even
dimensional proper scheme.  In section 3 we discuss various basic
material on kernels of Fourier-Mukai equivalences.  The main technical
tool is Proposition \ref{P:adjoint-love}, which will be used when
deforming kernels.  The results of these two sections are presumably
well-known to experts. 

As an application of the formalism of Mukai motives we prove Theorem \ref{T:horse} in section 4.

In section 5 we discuss the relationship between moduli of complexes
and Fourier-Mukai kernels.  This relationship is the key to the
deformation theory arguments that follow and appears never to have
been written down in this way.  The main result of this
section is Corollary  \ref{P:open in mother}.

Section 6 contains the key result for the whole paper (Theorem
\ref{T:derived torelli}).  This result should be viewed as a derived
category version of the classical Torelli theorem for \text{\rm K3}  surfaces. It
appears likely that this kind of reduction to the universal case via
moduli stacks of complexes should be useful in other contexts.

Using these deformation theory techniques we prove Theorem \ref{T:chicken} in section 7 and 
 statement (1) in Theorem \ref{T:mainthm} in section 8.

In section 9 we prove statement (2) in Theorem \ref{T:mainthm}.  Our
proof involves deforming to characteristic $0$, which in
particular is delicate for supersingular \text{\rm K3}  surfaces.

Finally there is an appendix containing a technical result about
versal deformations of polarized \text{\rm K3}  surfaces which is used in section
7. The main result of the appendix is Theorem \ref{P:A.1} concerning
the Picard group of the general deformation of a fixed \text{\rm K3}  surface from
characteristic $p$ to characteristic $0$.

\subsection{Notation}
For a proper smooth scheme $X$ over a field $k$ we write $K(X)$ for the Grothendieck group of vector bundles on $X$ and $A^*(X)_\Q$ for the Chow ring of algebraic cycles on $X$ modulo rational equivalence.    We write $\text{ch}:K(X)\rightarrow A^*(X)_\Q$ for the Chern character.

\subsection{Acknowledgments} Lieblich partially supported by the
Sloan Foundation, NSF grant DMS-1021444, and NSF CAREER grant
DMS-1056129, Olsson partially supported by NSF CAREER grant
DMS-0748718 and NSF grant DMS-1303173. We thank Andrew Niles, Daniel Huybrechts, Davesh Maulik, Richard Thomas, and two anonymous referees  for many helpful comments and error-correction.

\section{Mukai motive}
\label{sec:mukai-crystals}

\subsection{Mukai's original construction over $\C$: the Hodge realization \cite{M}}

Suppose $X$ is a smooth projective variety of even dimension $d=2\delta$.  The singular cohomology $\H^i(X,\Z)$ carries a natural pure Hodge structure of weight $i$, and the cup product defines a pairing of Hodge structures
$$\H^i(X,\Z)\times\H^{2d-i}(X,\Z)\to\H^{2d}(X,\Z)=\Z(-d),$$
where $\Z(-1)$ is the usual Tate Hodge structure of weight $2$.

Define the \emph{(even) Mukai-Hodge structure of $X$\/} to be the pure Hodge structure of weight $d$ given by
$$\widetilde H(X,\Z):=\bigoplus_{i=-\delta}^{\delta} \H^{d+2i}(X,\Z)(i).$$

The cup product and the identification $\H^{2d}(X,\Z)\cong\Z(-d)$ yield the \emph{Mukai pairing\/}
$$\widetilde H(X,\Z)\times\widetilde H(X,\Z)\to\Z(-d)$$
defined by the formula
$$
\langle (a_{-\delta }, a_{-\delta +1}, \dots, a_{\delta -1}, a_\delta ), (a'_{-\delta }, a'_{-\delta +1}, \dots, a'_{\delta -1}, a'_\delta )\rangle := \sum _{i=-\delta }^\delta (-1)^ia_{i}\cdot a_{-i}',
$$
where $a_i\in H^{d+2i}(X, \Z)(i)$. Similarly we get a rational Hodge structure $\widetilde H(X, \Q)$, replacing $\Z$ by $\Q$ in the preceding discussion. 

One of the main features of the Mukai-Hodge structure is its compatibility with correspondences.  In particular, let $Y$ be another smooth projective variety of dimension $d$. A perfect complex $P$ on $X\times Y$  induces a map of Hodge lattices
$$\Psi_P:\widetilde H(X,\Q)\to\widetilde H(Y,\Q)$$
given by adding the maps
$$
\Psi _P^{i,j}:H^{d +2i}(X,\Q)(i)\rightarrow H^{d  +2j}(Y,\Q)(j)
$$
defined as the composite
$$
\xymatrix{
H^{d +2i}(X,\Q)(i)\ar[r]^-{\text{pr}_1^*}& H^{d +2i}(X\times Y,\Q)(i)\ar[d]^-{\cup \text{\rm ch}^{ j-i+d}(P)}\\
H^{d +2j}(Y,\Q)(j)& H^{d +2j+2d}(X\times Y,\Q)(j+d)\ar[l]_-{\text{pr}_{2*}}.}
$$
  Note  that this map depends only on the image of $P$ in the Grothendieck group $K(X\times Y)$.  In general this map is defined only with rational coefficients, due to the presence of denominators in the Chern character, but for \text{\rm K3} surfaces it is defined integrally.

Instead of considering $\Psi _P$ it is standard to work with  $\Phi_P=\Psi_{P\sqrt{\Td_{X\times Y}}}$.  A discussion of the reasons for this normalization can be found in \cite[pp. 127--130]{H}. This modified map
$$\Phi_P:\widetilde H(X,\Q)\to\widetilde H(Y,\Q)$$
 has the advantage of making the transform compatible with inner products and Chern class maps.

Mukai's original work in \cite{M} was on the cohomology of \text{\rm K3}  surfaces.  For such a surface $X$, the Mukai-Hodge structure is $$\H^0(X,\Z)(-1)\oplus\H^2(X,\Z)\oplus\H^4(X,\Z)(1)$$
(colloquially rendered as ``place $\H^0$ and $\H^4$ in $\widetilde H^{1,1}$''), and the Mukai pairing takes the form
$$\langle(a,b,c),(a',b',c')\rangle=bb'-ac'-a'c.$$
Note that this pairing restricts to the usual pairing on $\H^2(X, \Z)$.
Moreover, the class $\sqrt{\Td_{X\times Y}}$ lies in $K(X\times Y)$ (i.e., it has integral components), so that for all pairs of \text{\rm K3}  surfaces $X$ and $Y$, any $P\in K(X\times Y)$ induces a map of rank 24 lattices
$$\Phi_P:\widetilde H(X,\Z)\to\widetilde H(Y,\Z).$$
As Mukai and Orlov proved in their seminal work (see \cite[3.3]{Orlov}), the Mukai-Hodge structure of a \text{\rm K3}  surface uniquely determines its derived category up to (non-canonical) equivalence.  In fact, they showed that the transcendental lattice alone suffices to determine the derived category.

The previous constructions generalize immediately to any Weil cohomology theory.  The main ones we consider are the following:

\subsection{Crystalline realization}

\begin{remark} For a survey of basic properties of crystalline cohomology see \cite{Ill}.
\end{remark}

Let $k$ be a perfect field of characteristic $p>0$, let $W$ be its ring of Witt vectors, and let $K$ denote the field of fractions of $W$.  For a proper smooth scheme $X/k$ let 
$H^i(X/K)$ denote
 the crystalline cohomology
$$
H^i(X/K):= H^i((X/W)_{\text{cris}}, \mls O_{X/W})\otimes _WK,
$$
an $F$-isocrystal over $K$.

Following standard conventions, let $K(1)$ denote the $F$-isocrystal whose underlying vector space is $K$, and whose Frobenius action is given by multiplication by $1/p$.  If $M$ is another isocrystal and $n$ is an integer we write $M(n)$ for the tensor product $M\otimes K(1)^{\otimes n}$ (with the usual convention that if $n$ is negative then $K(1)^{\otimes n}$ denotes the $-n$-th tensor power of the dual of $K(1)$).

For a proper smooth $k$-scheme $X$ there is a \emph{crystalline Chern character} 
$$
\text{ch}_{\text{cris}} :K(X)\rightarrow H^{2*}(X/K),
$$
obtained by composing the cycle class map $\eta :A^*(X)\rightarrow H^{2*}(X/K)$ (defined in \cite{G-M} and \cite{Gros}) with the Chern character $K(X)\rightarrow A^*(X)_\Q$.  See also \cite{Pet} for a discussion of chern classes in rigid cohomology.
 For an integer $i$ we write $\text{ch}_{\text{cris}}^i$ for the $2i$-th component of $\text{ch}_{\text{cris}}.$  Using the splitting principle to reduce to the case of a line bundle, one shows that for any integer $i$ and $E\in K(X)$ we have
$$
\varphi _X(\text{\rm ch}_{\text{\rm cris}}^i(E)) = p^i\text{\rm ch}_{\text{\rm cris}}^i(E).
$$

If $X/k$ is proper and smooth of even dimension $d = 2\delta $, set
$$
\widetilde H^i(X/K):= H^{d +2i}(X/K)(i), \ \ -\delta \leq i\leq \delta ,
$$
and define the \emph{Mukai $F$-isocrystal of $X/K$} to be the $F$-isocrystal
$$
\widetilde H(X/K):= \oplus _i\widetilde H^i(X/K).
$$

Then using the same formulas as in the Betti cohomology setting, there is a pairing $\langle \cdot, \cdot \rangle $ on $\widetilde H(X/K)$. Also if  $Y$ is a second smooth proper $k$-scheme of the same dimension as $X$ and $P\in K(X\times Y)$ is an object, then we get an isometry     of Mukai $F$-isocrystals
$$
\Phi _P:\widetilde H(X/K)\rightarrow \widetilde H(Y/K)
$$
using the crystalline Chern character multiplied by the square root of   the Todd class $\sqrt {\text{Todd}(X\times Y)}$.

\subsection{\'Etale realization}\label{S:etale realization}

Let $k$ be a field of characteristic $p$, and
fix a prime $\ell$ distinct from $p$.   Fix also a separable closure $k\hookrightarrow \bar k$, and let $G_k$ denote the Galois group of $\bar k$ over $k$. The \'etale realization of the Mukai motive is given by the  $G_k$-modules
$$
\widetilde H^i(X_{\bar k}, \Z_\ell ):= H^{d+2i}(X_{\bar k},\Z_\ell)(i), \ \ \widetilde H(X,\Z_\ell):=\oplus _i\widetilde H^i(X_{\bar k}, \Z_\ell ), \ \  -\delta\leq i\leq \delta,
$$
where $H^j(X_{\bar k}, \Z_\ell )$ denotes \'etale cohomology. 

The cycle class maps $A^i(X)\to\H^{2i}(X_{\bar k},\Z_\ell(i))$, Gysin maps, etc., yield identical funtorialities to the crystalline case, and   the usual formula yields a Mukai pairing.

When $X$ is defined over a finite field $\F_q$, the $q$th-power Frobenius gives an action of the arithmetic (and geometric) Frobenius on  $\widetilde H(X_{\bar k},\Z_\ell)$.  Given $X$, $Y$, and $P\in\D(X\times Y)$, all defined over $\F_q$, we get a Frobenius-invariant map
$$\Psi:\widetilde H(X_{\bar k},\Q_\ell)\to\widetilde H(Y_{\bar k},\Q_\ell).$$

In particular, when $\Psi $ is an equivalence the characteristic polynomials of Frobenius on the $\ell$-adic Mukai lattices are equal.

\subsection{De Rham realization}

For a smooth proper scheme $X$ over a field $k$ of characteristic $0$ let $H^s_{\text{dR}}(X/k)$ denotes the $s$-th de Rham cohomology group of $X$.  Recall that this is a filtered vector space with filtration $\text{Fil}_{\text{dR}}$ defined by the Hodge filtration.

Following standard conventions,  if $\mls V = (V, F^\bullet )$ is  a vector space with a decreasing filtration $F^\bullet$, define for an integer $n$ the \emph{$n$-th Tate twist of $\mls V$}, denoted $\mls V(n)$, to be the filtered vector space with the same underlying vector space $V$, but whose filtration in degree $s$ is given by $F^{s+n}$.

For $X/k$ smooth proper of dimension $d$ Poincar\'e duality then gives a perfect pairing in the category of filtered vector spaces
$$
H^i_{\text{dR}}(X/k)\otimes H^{2d-i}_{\text{dR}}(X/k)\rightarrow k(-d).
$$

If  $X$ is of even dimension $2\delta $, we set
$$
\widetilde H_{\text{dR}}(X/k):= \oplus _{i=-\delta }^\delta H^{d +2i}_{\text{dR}}(X/k)(i).
$$
This has an inner product, called the \emph{Mukai pairing}, taking values in $k(-d)$ defined by the same formula as in the Betti setting.

\begin{remark} In the case when $X$ is a surface  we have
$$
\widetilde H_{\text{dR}}(X/k) = H^0_{\text{dR}}(X/k)(-1)\oplus H^2_{\text{dR}}(X/k)\oplus H^4_{\text{dR}}(X/k)(1),
$$
the filtration is given by
$$
\text{Fil}^2 = \text{Fil}^2H^2_{\text{dR}}(X/k), \ \ \text{Fil}^1 = H^0_{\text{dR}}(X/k)\oplus \text{Fil}^1H^2_{\text{dR}}(X/k)\oplus H^4_{\text{dR}}(X/k),
$$
and $\text{Fil}^s$ is equal to $\widetilde H_{\text{dR}}(X/k)$ (resp.\ $0$) for $s\geq 0$ (resp.\ $s<2$).  Note that this also shows that $\text{Fil}^1$ is the orthogonal complement under the Mukai pairing of $\text{Fil}^2$.
\end{remark}

\subsection{Crystalline and de Rham comparison}

Consider now a  complete discrete valuation ring $V$ with perfect residue field $k$ and field of fractions $K$.    Let $W\subset V$ be the ring of Witt vectors of $k$, and let $K_0\subset K$ be its field of fractions.  Let $\mls X/V$ be a proper smooth scheme of even relative dimension $2\delta $, and let $\mls X_s$ (resp.\ $\mls X_\eta $) denote the closed (resp.\ generic) fiber.    We then have the Berthelot-Ogus comparison isomorphism \cite[2.4]{BO}
$$
H^*_{\text{cris}}(\mls X_s/K_0)\otimes _{K_0}K\rightarrow H^*_{\text{dR}}(\mls X_\eta /K).
$$
This isomorphism induces an isomorphism of graded $K$-vector spaces
$$
\sigma _{\mls X}:\widetilde H_{\text{cris}}(\mls X_s/K_0)\otimes _{K_0}K\rightarrow \widetilde H_{\text{dR}}(\mls X_\eta /K).
$$
Because the comparison isomorphism between crystalline cohomology and de Rham cohomology is compatible with cup product and respects the cohomology class of a point, the map $\sigma $ is compatible with the Mukai pairings on both sides.

Now suppose given two proper smooth $V$-schemes $\mls X$ and $\mls Y$ of the same even dimension, and let $X$ and $Y$ respectively denote their reductions to $k$.    Suppose further given a perfect complex $P$ on $X\times Y$ such that the induced map
$$
\Phi _P^{\text{cris}}:\widetilde H_{\text{cris}}(X/K_0)\rightarrow \widetilde H_{\text{cris}}(Y/K_0)
$$
is an isomorphism.    We then get an isomorphism
\begin{equation}\label{E:Frob}
\xymatrix{
\widetilde H_{\text{dR}}(\mls X_\eta )\ar[r]^-{\sigma _{\mls X}^{-1}}& \widetilde H_{\text{cris}}(X/K_0)\otimes K\ar[r]^-{\Phi _P^{\text{cris}}}& \widetilde H_{\text{cris}}(Y/K_0)\otimes K\ar[r]^-{\sigma _{\mls Y}}& \widetilde H_{\text{dR}}(\mls Y_\eta ).}
\end{equation}

\begin{defn}   The families $\mls X/V$ and $\mls Y/V$ are called \emph{$P$-compatible} if the composite morphism (\ref{E:Frob}) respects the Hodge filtrations.
\end{defn}

\subsection{Chow realization}

For a scheme  $X$ proper and smooth over a field $k$, let
$A^* (X)_{\text{\rm num}}$ denote the graded group of algebraic cycles on $X$ modulo numerical equivalence, and let $A^*(X)_{\text{\rm num}, \Q}$ denote $A^* (X)_{\text{\rm num}}\otimes \Q$.

If $X$ and $Y$ are two smooth proper $k$-schemes of the same even dimension $d = 2\delta $, and if $P\in \D(X\times Y)$ is a perfect complex,
then we can consider the class $\beta(P):=\chern(P)\cdot\sqrt{\Td_{X\times Y}}\in A^*(X\times Y)_{\text{\rm num}, \Q}$.  This class induces a map
$$
\Phi ^{A^*_{\text{\rm num}}}_P:A^*(X)_{\text{\rm num}, \Q}\to A^*(Y)_{\text{\rm num}, \Q},
$$
defined by the formula
$$
\Phi ^{A^*_{\text{\rm num}}}_P(\alpha ) = \text{pr}_{2*}(\text{pr}_1^*(\alpha )\cup \beta (P)).
$$

In the case when $k$ is a perfect field of positive characteristic,
the cycle class map defines maps
$$
\text{cl}_X:A^*(X)_{\text{\rm num}, \Q}\rightarrow \widetilde H(X/K), \ \ \text{cl}_Y:A^* (Y)_{\text{\rm num}, \Q}\rightarrow \widetilde H(Y/K)
$$

and

$$
\text{cl}_X:A^* (X)_{\text{\rm num}, \Q}\rightarrow \widetilde H(X,\Q_\ell), \ \ \text{cl}_Y:A^*(Y)_{\text{\rm num}, \Q}\rightarrow \widetilde H(Y,\Q_\ell).
$$

\begin{prop}\label{P:chow-compatibility}
The diagrams
 $$\xymatrix{A^*(X)_{\text{\rm num}, \Q}\ar[r]^{\Phi^{A^*(-)_{\text{\rm num}, \Q}}_P}\ar[d]^-{\text{\rm cl}_X} & A^*(Y)_{\text{\rm num}, \Q}\ar[d]^-{\text{\rm cl}_Y}\\
 \widetilde H(X/K)\ar[r]_{\Phi_P} &\widetilde H(Y/K)}
 $$
and
 $$\xymatrix{A^*(X)_{\text{\rm num}, \Q}\ar[r]^{\Phi^{A^*(-)_{\text{\rm num}, \Q}}_P}\ar[d]^-{\text{\rm cl}_X} & A^*(Y)_{\text{\rm num}, \Q}\ar[d]^-{\text{\rm cl}_Y}\\
 \widetilde H(X,\Q_\ell)\ar[r]_{\Phi_P} &\widetilde H(Y,\Q_\ell)}
 $$
commute.
\end{prop}
\begin{proof}
This follows from the fact that the cycle class map commutes with smooth  pullback, proper pushforward, and cup product.  In the \'etale context  a reference for this is \cite[\S 6]{Laumon} and in the crystalline case see \cite{G-M} or \cite[II 4.2]{Gros}.
\end{proof}

\begin{pg}\label{Para:filt}
 It will be useful to consider the codimension filtration $\text{Fil}_X^\bullet $ on
$$
A^*(X)_{\text{\rm num}, \Q} = \oplus _{i}A^i(X)_{\text{\rm num}, \Q}
$$
given by
$$
\text{Fil}_X^s := \oplus _{i\geq s}A ^i(X)_{\text{\rm num}, \Q}.
$$

If $X$ and $Y$ are smooth proper $k$-schemes, and $P\in \D(X\times Y)$ is a perfect complex, then we say that
$$
\Phi _P^{A^*(-)_{\text{\rm num}, \Q}}:A^*(X)_{\text{\rm num}, \Q}\rightarrow A^*(Y)_{\text{\rm num}, \Q}
$$
is \emph{filtered} if it preserves the codimension filtration.  We will also sometimes refer to  the functor
$$
\Phi _P:\D(X)\rightarrow \D(Y)
$$
as being filtered, meaning that $\Phi _P^{A^*(-)_{\text{\rm num}, \Q}}$ is filtered (this apparently abusive terminology is justified by the theorem of Orlov recalled in \ref{Orlov} below, which implies that $P$ is determined by the equivalence $\Phi _P:\D(X)\rightarrow \D(Y)$).
\end{pg}

Observe that in the case when $X$ and $Y$ are surfaces, we have
$$
\text{Fil}_X^0 = A^* (X)_{\text{\rm num}, \Q}, \text{Fil}_X^1 = A^1(X)_{\text{\rm num}, \Q}\oplus A^2(X)_{\text{\rm num}, \Q}, \ \text{Fil}_X^2 = A^2(X)_{\text{\rm num}, \Q},
$$
and $\text{Fil}_X^1 $ is the subgroup of elements orthogonal to $\text{Fil}_X^2$.  Therefore in the case of surfaces, $\Phi ^{A^*(-)_{\text{\rm num}, \Q}}_P$ is filtered if and only if
$$
\Phi ^{A^*(-)_{\text{\rm num}, \Q}}_P(\text{Fil}_X^2) = \text{Fil}_Y^2.
$$

\subsection{Mukai vectors of perfect complexes}

Let $X$ be a smooth projective geometrically connected scheme over a field $k$.

\begin{defn}\label{D:Mukai vector}
Given a perfect complex $E\in\D(X)$, the \emph{Mukai vector\/} of $E$ is
$$
v(E):=\chern(E)\sqrt{\Td_X}\in A^* (X)_{\text{\rm num}, \Q}
$$
\end{defn}

In the case when $X$ is a \text{\rm K3}  surface, the Mukai vector of a complex $E$ is given by (see for example \cite[p. 239]{H})
$$
v(E) = (\text{rk}(E), c_1(E), \text{rk}(E)+c_1(E)^2/2-c_2(E)).
$$
In particular, using that the Todd class of the tangent bundle of $X$ is $(1, 0, 2)$, one gets by Grothendieck-Riemann-Roch that  for two objects $E, F\in \D(X)$ we have
\begin{equation}\label{E:chiformula}
\langle v(E), v(F)\rangle = -\chi (E, F).
\end{equation}
As a consequence, if $E$ is a simple torsion free sheaf on a \text{\rm K3}  surface $X$, the universal deformation of $E$ (keeping $X$ fixed) is formally smooth of dimension $v(E)^2-2$,  hinting that the Mukai lattice captures the numerology needed to study moduli and deformations.  (A review of the standard results in this direction may be found in section \ref{modulidef} below.)

The compatibility of the Chow realization with the crystalline, \'etale, and de Rham realizations yields Mukai vectors in each of those realizations, satisfying the same rule.

\begin{remark} In the above we work with realizations in the various cohomology theories.  One of course expects there to be an underlying motive whose realizations are given as above.  A precise definition of such a motive, however, seems to require the standard projectors in cohomology $H^*(X)\rightarrow H^i(X)$ to be given by morphisms in the category of motives.  This is closely related to  the K\"{u}nneth standard conjecture, denoted $C(X)$ in \cite[p. 14]{Kleiman}.  Without this one cannot relate $H^*(X)$ to $\widetilde H^*(X)$.
\end{remark}

\section{Kernels of Fourier-Mukai equivalences}

\subsection{Generalities}
Let $X$ and $Y$ be proper smooth schemes over a field $k$. For a perfect complex $P$ on $X\times Y$, consider the functor
$$
\Phi _D^P:\D(X)\rightarrow \D(Y)
$$
given by
$$
\Phi _D^P(K):= \text{pr}_{2*}(P\otimes ^{\mathbb{L}}\text{pr}_1^*K).
$$

Let $P^\vee  $ denote the complex
$$
P^\vee  := R\mls Hom(P, \mls O_{X\times Y}),
$$
which we view as a perfect complex on $Y\times X$ (switching the factors).  Let $\omega _X$ (resp.\ $\omega _Y$) denote the highest exterior power of $\Omega ^1_X$ (resp.\ $\Omega _Y^1$).

Let
\begin{equation}\label{E:adjoints}
G:\D(Y)\rightarrow \D(X) \ \ (\text{resp.} \ H:\D(Y)\rightarrow \D(X))
\end{equation}
denote
$$
\Phi _D^{P^\vee  \otimes \pi _Y^*\omega _Y[\text{dim}(Y)]} \ \ (\text{resp.} \ \Phi _D^{P^\vee  \otimes \pi _X^*\omega _X[\text{dim}(X)]}),
$$
where $\pi _X$ and $\pi _Y$ denote the projections from $X\times Y$.  From Grothendieck duality one gets:

\begin{prop}[{\cite[4.5]{B}}]  The functor $G$ (resp.\ $H$) is left adjoint (resp.\ right adjoint) to $\Phi _D^P$.
\end{prop}

For later use let us recall how these adjunction maps
$$
\eta:G\circ \Phi _D^P\rightarrow \text{id}, \ \ \epsilon :\text{id}\rightarrow H\circ \Phi _D^P
$$
are obtained.

In general if $X$, $Y$, and $Z$ are proper smooth $k$-schemes, $P$ is a perfect complex on $X\times Y$, and $Q$ is a perfect complex on $Y\times Z$, then the composite functor
$$
\xymatrix{
\D(X)\ar[r]^-{\Phi _D^P}& \D(Y)\ar[r]^-{\Phi _D^Q}& \D (Z)}
$$
is equal to
$$
\Phi _D^{\gamma _{X\times Z*}(\gamma _{X\times Y}^*(P)\otimes ^{\mathbb{L}}\gamma _{Y\times Z}^*Q)},
$$
where $\gamma _{X\times Z}$, $\gamma _{X\times Y}$, and $\gamma _{Y\times Z}$ are the various projections from $X\times Y\times Z$.

In particular, taking $Z = X$ and $Q = P^\vee  \otimes \pi _Y\omega _Y[\text{dim}(Y)]$, we get that the composition $G\circ \Phi _D^P$ is equal to
$$
\Phi _D^{\gamma  _{X\times X*}(\gamma  _{X\times Y}^*(P)\otimes \gamma  _{Y\times X}^*(P^{\vee  }\otimes \sigma  _Y^*\omega _Y[\text{dim}(Y)])}.
$$
The adjunction
$$
\eta :G\circ \Phi _D^P\rightarrow \text{id}
$$
 is realized by a map
\begin{equation}\label{E:bareta}
\bar \eta :\gamma  _{X\times X*}(\gamma  _{X\times Y}^*(P)\otimes \gamma  _{Y\times X}^*(P^{\vee  }\otimes \sigma  _Y^*\omega _Y[\text{dim}(Y)]))\rightarrow \Delta _{*}\mls O_X,
\end{equation}
where we note that
$$
\Phi _D^{\Delta _{*}\mls O_X} = \text{id}.
$$

This map $\bar \eta $ is adjoint to the map
$$
\xymatrix{
\Delta ^*\gamma  _{X\times X*}(\gamma  _{X\times Y}^*(P)\otimes \gamma  _{Y\times X}^*(P^{\vee  }\otimes \sigma  _Y^*\omega _Y[\text{dim}(Y)]))\ar[d]^-\simeq \\
\pi _{X*}(P\otimes P^\vee  \otimes \pi _Y^*\omega _Y[\text{dim}(Y)])\ar[d]^-{P\otimes P^\vee  \rightarrow \text{id}}\\
\pi _{X*}\pi _Y^*\omega _Y[\text{dim}(Y)]\ar[d]^-\simeq \\
R\Gamma (Y, \omega _Y[\text{dim}(Y)])\otimes _k\mls O_X\ar[d]^-{\text{tr}\otimes 1}\\
\mls O_X.}
$$

Similarly, the composite $H\circ \Phi _D^P$ is induced by the perfect complex
$$
\gamma  _{X\times X*}(\gamma  _{X\times Y}^*P\otimes \gamma _{Y\times X}^* (P^\vee  \otimes \pi _X^*\omega _X[\text{dim}(X)]))
$$
on $X\times X$.  There is a natural map
\begin{equation}\label{E:barepsilon}
\bar \epsilon :\Delta _*\mls O_X\rightarrow \gamma  _{X\times X*}(\gamma  _{X\times Y}^*P\otimes \gamma _{Y\times X}^* (P^\vee  \otimes \pi _X^*\omega _X[\text{dim}(X)]))
\end{equation}
which induces the adjunction map
$$
\epsilon :\text{id}\rightarrow H\circ \Phi _D^P.
$$

  The map $\bar \epsilon $ is obtained by noting that
$$
\Delta ^!\gamma  _{X\times X*}(\gamma  _{X\times Y}^*P\otimes \gamma _{Y\times X}^* (P^\vee  \otimes \pi _X^*\omega _X[\text{dim}(X)]))\simeq \pi _{X*}(P\otimes P^\vee  ),
$$
so giving the map $\bar \epsilon $ is equivalent to giving a map
\begin{equation}\label{E:adjointmap}
\ms O_X\to\pi_\ast(P\tensor P^\vee ),
\end{equation}
and this is adjoint to a map
$$\ms O_{X\times Y}\to P\tensor P^\vee .$$
Taking the natural scaling map for the latter gives rise to the desired map $\bar\epsilon$.

\begin{prop}\label{P:adjoint-love} The functor $\Phi _D^P$ is an equivalence if and only if the maps  (\ref{E:bareta}) and (\ref{E:barepsilon}) are isomorphisms.
\end{prop}
\begin{proof}
For a closed point $x\in X(k)$, let $P_x\in D(Y)$ denote the object obtained by pulling back $P$ along
$$
\xymatrix{
Y\simeq \Sp (k)\times Y\ar[r]^-{i_x\times \text{id}}& X\times Y,}
$$
where $i_x:\Sp (k)\hookrightarrow X$ is the closed immersion corresponding to $x$. If we write $\mls O_x$ for $i_{x*}\mls O_{\Sp (k)}$, then we have
$$
P_x = \Phi _D^P(\mls O_x).
$$

\begin{lem}Let $\rho :K\rightarrow K'$ be a morphism in $D(X\times Y)$.  Then the induced morphism of functors $\Phi (\rho  ):\Phi _D^K\rightarrow \Phi _D^{K'}$ is an isomorphism if and only if $\rho $ is an isomorphism.
\end{lem}
\begin{proof} The `if' direction is immediate.  For the `only if' direction, observe that if $x\in X(k)$ is a point, then the condition that $\Phi (\rho )$ is an isomorphism implies that the map $\rho _x:K_x\rightarrow K_{x}'$ is an isomorphism.  By Nakayama's lemma (for perfect complexes) this implies that $\rho $ is an isomorphism.
\end{proof}

This implies Proposition \ref{P:adjoint-love}, since it implies that the maps (\ref{E:bareta}) and (\ref{E:barepsilon}) are isomorphisms if and only if the adjunction maps for the adjoint pairs $(G, \Phi _D^P)$ and $(\Phi _D^P, H)$ are isomorphisms, which is equivalent to $\Phi _D^P$ being an equivalence.
\end{proof}

 In what follows we call an object $P\in D(X\times Y)$ a \emph{Fourier-Mukai kernel} (or \emph{FM-kernel}) if the functor $\Phi ^P_D$ is an equivalence.  

Let us also recall the following three results which will be relevant in the following discussion.

\begin{prop}\label{P:3.5} Let $P\in D(X\times Y)$ be a FM kernel, and let $x, x'\in X(k)$ be points.

(i) If $x \neq x'$ then
$\text{\rm RHom}(P_x, P_{x'}) = 0.$

(ii) There is a natural isomorphism $T_xX\simeq \text{\rm Ext}^1(P_x, P_x)$ which induces an isomorphism of algebras
$$
\Lambda ^\bullet T_xX\simeq \oplus _i\text{\rm Ext}^i(P_x, P_x).
$$
\end{prop}
\begin{proof} Indeed since $P$ is a FM kernel, we have
$$
\text{RHom}(P_x, P_{x'})\simeq \text{RHom}(\mls O_x, \mls O_{x'}).
$$
This implies (i) and also reduces the proof of (ii) to the  calculation of the right side.  There is a resolution in the category of $\mls O_{X, x}$-modules
$$
0\rightarrow \mathfrak{m}\rightarrow \mls O_{X, x}\rightarrow \mls O_x\rightarrow 0,
$$
where $\mathfrak{m}$ denotes the maximal ideal in $\mls O_{X, x}$, 
which upon applying $\text{RHom}_{\mls O_{X, x}}(-, \mls O_x)$ induces a morphism
$$
T_xX = (\mathfrak{m}/\mathfrak{m}^2)^\vee \rightarrow \text{Ext}^1(\mls O_x, \mls O_x).
$$
This induces a morphism of algebras as in (ii). The verification that it is an isomorphism is a standard exercise in Koszul resolutions (see for example \cite[Exercise 4.5.6]{Weibel}).
\end{proof}

\begin{thm}[Orlov {\cite[2.2]{Orlov}}]\label{Orlov} Let $X$ and $M$ be smooth projective schemes over a field $k$, and let
$$
F:\D(X)\rightarrow \D(M)
$$
be an equivalence of triangulated categories.  Then $F = \Phi ^P$ for  a perfect complex $P$ on $X\times M$, and the complex $P$ is unique up to isomorphism.
\end{thm}

\begin{remark} A generalization of Orlov's theorem has been obtained by Canonaco and Stellari \cite[Theorem 1.1]{CS}, but the above suffices for the purposes of this paper.
\end{remark}

\begin{remark} In what follows we refer to an equivalence of triangulated categories $F:\D(X)\rightarrow \D(M)$ as a \emph{Fourier-Mukai transform}.  One could also consider more general additive functors, but in this paper we restrict our attention to equivalences. Two smooth projective $k$-schemes $X$ and $M$ are called \emph{Fourier-Mukai partners} (or \emph{FM-partners}) if there exists a Fourier-Mukai kernel $P\in D(X\times M)$.  By Orlov's theorem, $X$ and $M$ are FM-partners if and only if there exists an equivalence of triangulated categories $D(X)\simeq D(M)$.
\end{remark}

\begin{prop}
 If $X$ is a \text{\rm K3}  surface over $k$ and $Y$ is a smooth projective $k$-scheme such that there is an equivalence $\D(X)\simto\D(Y)$ then $Y$ is a \text{\rm K3}  surface.
\end{prop}
\begin{proof}
See for example  \cite[Corollary 10.2]{H} (though the proof  given there is  in the characteristic $0$ setting it generalizes immediately to arbitrary characteristic).
\end{proof}

We end this section with a brief review of some standard kernels.

\subsection{Tensoring with line bundles}

Let $k$ be a field and $X/k$ a smooth projective scheme.  Let $L$ be a line bundle on $X$.  Then the equivalence of triangulated categories
$$
\D(X)\rightarrow \D(X), \ \ K\mapsto K\otimes ^{\mathbb{L}}L
$$
is induced by the kernel $P:= \Delta _*L$ on $X\times X$, where $\Delta :X\rightarrow X\times X$ is the diagonal.  In the case when $X$ is a surface the corresponding action on the Mukai motive is given by the map sending
$(a, b, c)\in A^* (X)_{\text{\rm num}, \Q}$ to
\begin{equation}\label{E:Lformula}
(a, b+ac_1(L), c+b\cdot c_1(L)+ac_1(L)^2/2).
\end{equation}

\subsection{Spherical twist \cite[Chapter 8]{H}}

Recall the following definition (see e.g. \cite[Definition 8.1]{H}).
\begin{defn}\label{D:spherical}
A perfect complex $E\in\D(X)$ is \emph{spherical\/} if
\begin{enumerate}
\item $E\ltensor\omega_X\cong E$
\item $\Ext^i(E,E)=0$ unless $i=0$ or $i=\dim(X)$, and in those cases we have $\dim\Ext^i(E,E)=1$
\end{enumerate}
\end{defn}
In other words, $\rshom(E,E)$ has the cohomology of a sphere.  A standard example is given by the structure sheaf of a $(-2)$-curve in a \text{\rm K3}  surface.

The trace map
$$\rshom(E,E)\to\ms O_X$$
defines a morphism
$$t:\L p^\ast\ms E^\vee\tensor \L q^\ast\ms E\to\R\Delta_\ast\ms O_X$$
in $\D(X\times X)$, where $p$ and $q$ are the two projections.  Define $P_E$ to be the cone over $t$.  The following result dates to work of Kontsevich, Seidel and Thomas.

\begin{thm}[{\cite[Proposition 8.6]{H}}]\label{T:spherical equivalence}
The transform $$T_E:\D(X)\to\D(X)$$ induced by $P_E$ is an equivalence of derived categories.
\end{thm}

This transform is called a \emph{spherical twist} \cite[8.3]{H}.  

\begin{prop}[{\cite[Lemma 8.12]{H}}]\label{P:spherical reflection}
Suppose $X$ is a \text{\rm K3}  surface and $P_E$ is the complex associated to a spherical object $E\in\D(X)$ as above.  Let $H$ be any realization of the Mukai motive described in the preceding sections (\'etale, crystalline, de Rham, Chow) and let $v\in H$ be the Mukai vector of $P_E$.  Then $v^2=2$ and the induced map
$$\Phi_{P_E}:H(X)\to H(X)$$
is the reflection in $v$:
$$\Phi_{P_E}(x)=x-(x\cdot v)v.$$
\end{prop}

\subsection{Moduli spaces of vector bundles}\label{modulidef}

Let $X$ be a \text{\rm K3}  surface over a field $k$.  One of Mukai and Orlov's wonderful discoveries is that one can produce Fourier-Mukai equivalences between $X$ and moduli spaces of sheaves on $X$ by using tautological sheaves.

If $S$ is a $k$-scheme and $E$ is a locally finitely presented quasi-coherent sheaf on $X\times S$ flat over $S$, then we get a function on the points of $S$ to $A^* (X)_{\text{\rm num}, \Q}$ by sending a point $s$ to the Mukai vector of the restriction $E_s$ of $E$ to the fiber over $s$.  This function is a locally constant function on $S$, and so if $S$ is connected it makes sense to talk about the Mukai vector of $E$, which is defined to be the  Mukai vector of $E_s$ for any $s\in S$.

For any ample class $h$ on $X$, let $\mls M_h$ denote the algebraic
stack of Gieseker-semistable sheaves on $X$, where semistability is defined using $h$.  A good summary of the standard results on these moduli spaces may be found in  \cite[Section 10.3]{H}, with a more comprehensive treatment in \cite{HL} and some additional non-emptiness results in \cite{M}. For a discussion of semistable sheaves in positive characteristic see Langer's paper  \cite{Langer}. If we fix a vector $v\in A^* (X)_{\text{\rm num}, \Q}$, we then get an open and closed substack $\mls M_h(v)\subset \mls M_h$ classifying semistable sheaves on $X$ with Mukai vector $v$.  Since the Mukai vector of a sheaf determines the Hilbert polynomial, the stack $\mls M_h(v)$ is an algebraic stack of finite type over $k$.  In fact, it is a GIT quotient stack with projective GIT quotient variety (a reference for this point of view on moduli of sheaves, and more generally twisted sheaves,  is \cite[Section 2.3.3]{LL}).

\begin{thm}[{\cite[Theorem 5.1]{M} and  \cite[Section 4.2]{Orlov2}}] \label{T:modulitheorem}Let $X$ be a \text{\rm K3}  surface over a field $k$.

\begin{enumerate}
\item Let $v\in A^*(X)_{\text{\rm num}, \Q}$ be a primitive element with  $v^2=0$ (with respect to the Mukai pairing)  and positive degree $0$ part.  Then $\ms M_h(v)$ is non-empty.
\item If, in addition, there is a complex $P\in D(X)$ with Mukai vector $v'$ such that $\langle v,v'\rangle=1$ then every semistable sheaf with Mukai vector $v$ is locally free and geometrically stable  \cite[Remark 6.1.9]{HL}, in which case $\mls M_h(v)$ is a $\m_r$-gerbe over a smooth projective \text{\rm K3}  surface $M_h(v)$ such that the associated $\G_m$-gerbe is  trivial (in older language, a tautological family exists on $X\times M_h(v)$).
\item A tautological family $\ms E$ on $X\times M_h(v)$ induces a Fourier-Mukai equivalence
$$\Phi_{\ms E}:\D(M_h(v))\to\D(X),$$
and thus in the case when $k = \mathbb{C}$ an isomorphism of Mukai-Hodge lattices.
\end{enumerate}
Finally, if $k = \mathbb{C}$, then any FM partner of $X$ is of this form.
\end{thm}

\begin{remark}
  While the non-emptiness uses the structure of analytic moduli of \text{\rm K3} 
  surfaces (and deformation to the Kummer case), it still holds over
  any algebraically closed field. One can see this by lifting $v$ and
  $h$ together with $X$ using Deligne's theorem and then
  specializing semistable sheaves using Langton's theorem \cite[Theorem on p. 99]{Langton}.
\end{remark}

\begin{remark}
  Using the results of \cite{huy}, one can restrict to working with
  moduli of locally free slope-stable sheaves from the start.
\end{remark}

\begin{remark}  All the statements in \ref{T:modulitheorem} can be verified over an algebraic closure of $k$, except for the triviality of the $\mathbb{G}_m$-gerbe $\mls M\rightarrow M_h(v)$ associated to the $\m _r$-gerbe $\mls M_h(v)\rightarrow M_h(v)$ in (2).  The triviality of this gerbe can be seen as follows. Let $\alpha \in H^2(M_h(v), \mathbb{G}_m)$ be the class in the cohomological Brauer group corresponding to $\mls M$,   and let $\mls E$ be the universal bundle on $X\times \mls M$.  Then for every geometric point $\bar y\rightarrow X\times \mls M$ the stabilizer action of $\mathbb{G}_m$ on $\mls E(\bar y)$ is the standard action. Consider the complex $K:= R\text{pr}_{2*}(\text{pr}_1^*(P^\vee [1])\otimes \mls E)$ on $\mls M$.  It follows from \eqref{E:chiformula} that this complex has rank $1$ and therefore the determinant $L:= \text{det}(K)$ is an $\alpha $-twisted sheaf on $\mls M$ in the sense of \cite{L2}.  It follows that the gerbe is trivial (in fact a trivialization is provided by the complement of the zero section in the total space of $L$ which maps isomorphically to $M_h(v)$).
\end{remark}

\section{Zeta functions of FM partners over a finite field}
\label{S:zeta}

In this section we address a question due to Musta\c{t}\u{a} and communicated to us by Huybrechts: do Fourier-Mukai partners over a finite field have the same zeta function?

\begin{thm}\label{T:point count}
 Suppose $X$ and $Y$ are \text{\rm K3}  surfaces over a finite field $k$.  If there is an equivalence $\D(X)\simto\D(Y)$ of $k$-linear derived categories then for all finite extensions $k'/k$ we have that $$|X(k')|=|Y(k')|.$$
In particular, $\zeta_X=\zeta_Y$.
\end{thm}
\begin{proof}
 By \cite[Theorem 3.2.1]{Orlov}, there is a kernel $P\in\D(X\times Y)$ giving the equivalence. The Leftschetz fixed-point formula in crystalline cohomology shows that it is enough to see that the trace of Frobenius acting on $\Hcris^2$ is the same.  The kernel $P$ induces an isomorphism of $F$-isocrystals $$\widetilde H(X/K)\simto\widetilde H(Y/K).$$
Thus, the trace of Frobenius on both sides is the same.  On the other hand, it follows from the definition of the Mukai crystal that
$$\Tr(F | \widetilde H(X/K))=\Tr(F | \Hcris^2(X/K))+2p$$
and similarly for $Y$.  Thus
$$\Tr(F | \Hcris^2(X/K))=\Tr(F | \Hcris^2(Y/K)),$$
giving the desired result.
\end{proof}

\section{Fourier-Mukai transforms and moduli of complexes}

In this section, we extend the philosophy of Mukai and Orlov and show that FM kernels on $X\times Y$  correspond to certain open immersions of $Y$ into a moduli space of complexes on $X$.  We start by reviewing the basic results of \cite{L} on moduli of complexes.

Given a proper morphism of finite presentation between schemes $Z\to S$, define a category fibered in groupoids as follows: the objects over an $S$-scheme $T\to S$ are objects $E\in\D(\ms O_{Z_T})$ of the derived category of sheaves of $\mls O_{Z_T}$-modules on $Z_T$ such that
\begin{enumerate}
\item for any quasi-compact $T$-scheme $U\to T$ the complex $E_U:=\L p^\ast E$ is bounded, where $p:Z_U\to Z_T$ is the natural map (``$E$ is relatively perfect over $T$'');
\item for every geometric point $s\to S$, we have that $\Ext^i(E_s,E_s)=0$ for all $i<0$ (``$E$ is universally gluable'').
\end{enumerate}

By  \cite[theorem on p. 2]{L}  this category fibered in groupoids is an Artin stack locally of finite presentation over $S$.  This stack is denoted $\ms D_{Z/S}$ (or just $\ms D_Z$ when $S$ is understood).

\begin{remark}\label{R:closedonly} In \cite[Proposition 2.1.9]{L}, it is proven that if $f:Z\to S$ is flat then a relatively perfect complex $E$ is universally gluable if and only if the second condition holds after base change to geometric points of $S$, and it is straightforward to see that it suffices to check at closed points of $S$ when the closed points are everywhere dense (e.g., $S$ is of finite type over a field).

Furthermore, suppose $f:Z\rightarrow S$ is smooth and that $S$ is of finite type over a field $k$. In this case it suffices to verify both conditions for geometric points lying over closed points of $S$ (i.e., we need not even assume $E$ is relatively perfect in order to get a fiberwise criterion).  This can be seen as follows.   First of all we may without loss of generality assume that $k$ is algebraically closed. Suppose  condition (2) holds for all closed points, and let $\bar \eta \rightarrow S$ be a geometric point lying over an arbitrary point $\eta \in S$.  Let $Z\hookrightarrow S$ be the closure of $\eta $ with the reduced structure.  Replacing $S$ by $Z$ we can assume that $S$ is integral with generic point $\eta $, and after shrinking further on $S$ we may also assume that $S$ is smooth over $k$.  Consider the complex
\begin{equation}\label{E:Homcomplex}
Rf_*\mls RHom(E, E)
\end{equation}
on $S$.  The restriction of this complex to $\bar \eta $ computes $\Ext ^i(E_{\bar \eta }, E_{\bar \eta })$, so it suffices to show that (\ref{E:Homcomplex}) is in $D^{\geq 0}(S)$.  Since $S$, and hence also $Z$ is smooth, the complex $\mls RHom(E, E)$ is a bounded complex on $S$.  Now by standard base change results, we can after further shrinking arrange that the sheaves $R^if_*\mls RHom(E, E)$ are all locally free on $S$, and that their formation commute with arbitrary base change on $S$.  In this case by our assumptions for $i<0$ these sheaves must be zero since their fibers over any closed point of $S$ is zero.
\end{remark}

Recall from \cite[4.3.1]{L} that an object $E\in \ms D_{Z/S}(T)$ over some $S$-scheme $T$ is called \emph{simple} if the natural map of algebraic spaces
$$
\mathbb{G}_m\rightarrow \mls Aut(E)
$$
is an isomorphism.  By \cite[4.3.2 and 4.3.3]{L}, the substack $s\ms D_{Z/S}\subset \ms D_{Z/S}$ classifying simple objects is an open substack, and in particular $s\ms D_{Z/S}$ is an algebraic stack.  Moreover, there is a natural map
$$
\pi :s\ms D_{Z/S}\rightarrow \sMom _{Z/S}
$$
from $s\ms D_{Z/S}$ to an algebraic space $\sMom_{Z/S}$ locally of finite presentation over $S$ which realizes $s\ms D_{Z/S}$ as a $\mathbb{G}_m$-gerbe.

Fix smooth projective varieties $X$ and $Y$ over a field $k$.
Let $\ms F\!\!\!\!\ms F$ be the groupoid of  complexes $P\in\D(X\times Y)$ such that the transform $$\Phi_P:\D(X)\to\D(Y)$$ is fully faithful. 

\begin{lem}\label{L:combine} Let $P\in \ms F\!\!\!\ms F$ be an object.
\begin{enumerate}
\item  [(i)] The complex $P$ is a simple object of $\ms D_{Y}(X)$. 
\item  [(ii)] The map $\bar \mu _P:X\rightarrow \sMom _Y$ obtained by composing the map $\mu _P:X\rightarrow s\ms D_Y$ defined by $P$ with $\pi :s\ms D_Y\rightarrow \sMom _Y$ is an open immersion.
\end{enumerate}
\end{lem}
\begin{remark} In what follows, we write $\mu _P:X\rightarrow s\ms D_Y$ for the morphism defined by an object $P\in \mc F\!\!\!\ms F$, as in statement (ii).
\end{remark}
\begin{proof}[Proof of \ref{L:combine}]
For (i) note that since $X$ is smooth the first condition ($P$ has finite Tor-dimension over $X$) is automatic.  It therefore suffices (using Remark \ref{R:closedonly}) to show that for any geometric point $\bar x\rightarrow X$ lying over a closed point of $X$ we have
$$
\text{Ext}^i(P_{\bar x}, P_{\bar x}) = 0
$$
for $i<0$, where $P_{\bar x}$ denotes the pullback of $P$ along
$$
\bar x\times Y\rightarrow X\times Y.
$$
This follows from proposition \ref{P:3.5} (i).

To prove (ii)  it suffices to show that the map $\bar \mu _P$  is an \'etale monomorphism.
For this  it suffices in turn to show  that for distinct closed points $x_1,x_2\in X(k)$ the objects $P_{x_1}$ and $P_{x_2}$ are not isomorphic, and that (see \cite[IV.17.11.1]{EGA})
$\bar \mu _P$ induces an isomorphism on tangent spaces.

To see that $P_{x_1}$ and $P_{x_2}$ are not isomorphic for $x_1\neq x_2$, note that since
 $\Phi_P$ is fully faithful, it yields isomorphisms
 $$
 \Ext^i_X(k(x_1),k(x_2))\simto\Ext^i_Y(P_{x_1},P_{x_2})
 $$
for all $i$.  Thus, if $x_1\neq x_2$ we have that
$$
\Hom_Y(P_{x_1},P_{x_2})=0,
$$
 so that $P_{x_1}\not\cong P_{x_2}$.

Next we show that $\bar \mu _P$ induces an isomorphism on tangent spaces.

Recall from \cite[3.1.1]{L} the following description of the tangent space to $\sMom _Y$ at a point corresponding to a complex $E$.  First of all since
$$
\pi :\ms D _Y\rightarrow \sMom _Y
$$
is a $\mathbb{G}_m$-gerbe, the tangent space $T_{[E]}\sMom _Y$ to $\sMom _Y$ at $E$ is given by the set of isomorphism classes of pairs $(E', \sigma )$, where $E'\in \ms D_Y(k[\epsilon ])$ and $\sigma :E'\otimes ^{\mathbb{L}}k\rightarrow E$ is an isomorphism.  Tensoring the exact sequence
$$
0\rightarrow \mls O_Y\otimes _k(\epsilon )\rightarrow \mls O_{Y_{k[\epsilon ]}}\rightarrow \mls O_Y\rightarrow 0
$$
with $E'$ we see that a deformation $(E', \sigma )$ of $E$ to $k[\epsilon ]$ induces a distinguished triangle
$$
\xymatrix{
E\otimes (\epsilon )\ar[r]& E'\ar[r]& E\ar[r]^-{\partial _{(E', \sigma )}}&E\otimes (\epsilon )[1].}
$$
In this way we obtain a map
\begin{equation}\label{E:tangentmap}
T_{[E]}\sMom _Y\rightarrow \Ext ^1_{Y_{k[\epsilon ]}}(E, (\epsilon )\otimes E), \ \ \ (E', \sigma )\mapsto [\partial _{(E', \sigma )}].
\end{equation}
This map is a morphism of $k$-vector spaces.  It is injective as the isomorphism class of $(E', \sigma )$ can be recovered from $\partial _{(E', \sigma )}$ by taking a cone of $\partial _{(E', \sigma )}$  and rotating the resulting triangle.

The image of (\ref{E:tangentmap}) can be described as follows.   The usual derived adjunction formula gives an isomorphism
$$
\mls RHom _{\mls O_{Y_{k[\epsilon ]}}}(E, (\epsilon )\otimes E)\simeq \mls RHom _{\mls O_Y}(E\otimes _k(k\otimes _{k[\epsilon ]}^{\mathbb{L}}k), (\epsilon )\otimes E).
$$
Now we have
$$
k\otimes _{k[\epsilon ]}k\simeq \oplus _{i\geq 0}k[i],
$$
so this gives
$$
\Ext ^1_{Y_{k[\epsilon ]}}(E, (\epsilon )\otimes E)\simeq \oplus _{i\geq 0}\Ext ^{1-i}_{Y}(E, (\epsilon )\otimes E).
$$
Since $E$ is universally gluable this reduces to an exact sequence
$$
0\rightarrow \Ext ^1_{Y}(E, (\epsilon )\otimes E)\rightarrow \Ext ^1_{Y_{k[\epsilon ]}}(E, (\epsilon )\otimes E)\rightarrow \text{Hom}_Y(E, (\epsilon )\otimes E)\rightarrow 0.
$$
As explained in \cite[proof of 3.1.1]{L}, the image of (\ref{E:tangentmap}) is exactly
$$
\Ext ^1_{Y}(E, (\epsilon )\otimes E)\subset  \Ext ^1_{Y_{k[\epsilon ]}}(E, (\epsilon )\otimes E).
$$

Now taking $E = P_x$ for a closed point $x\in X(k)$, we get by applying the fully faithful functor $\Phi ^{P}$ an isomorphism
$$
\xymatrix{
\Ext ^1_X(k(x), (\epsilon )\otimes _kk(x))\ar[r]^-{\Phi ^P}& \Ext ^1_Y(P_x, (\epsilon )\otimes _kP_x)\ar[r]^-{\simeq }& T_{[P_x]}\sMom _Y.}
$$
On the other hand, applying $\Hom _X(-, k(x))$ to the short exact sequence
$$
0\rightarrow I_x\rightarrow \mls O_X\rightarrow k(x)\rightarrow 0,
$$
where $I_x$ denotes the coherent sheaf of ideals defining $x$, we get an exact sequence
$$
\Hom _X(\mls O_X, k(x))\rightarrow \Hom _X(I_x, k(x))\rightarrow \Ext ^1_X(k(x), k(x))\rightarrow 0.
$$
Since any morphism $\mls O_X\rightarrow k(x)$ factors through $k(x)$, this gives an isomorphism
$$
T_xX = \Hom (I_x/I_x^2, k(x))\simeq \Ext ^1_X(k(x), k(x)).
$$
Putting it all together we find an isomorphism
$$
T_xX\simeq T_{[P_x]}\sMom _Y.
$$
We leave to the reader the verification that this isomorphism is the map induced by $\bar \mu _P$, thereby completing the proof that $\bar \mu _P$ is \'etale.
\end{proof}

\begin{notation}
 Let $s\ms D_{Y}(X)^\circ $ be the groupoid of morphisms $$\mu :X\to s\ms D _Y$$ such that the composed map
$$\xymatrix{X\ar[r]^-{\mu }\ar[dr]_{\overline\mu} & \Mom_Y\ar[d]\\
&\sMom_Y}$$
is an open immersion.  
\end{notation}

\begin{cor}\label{P:open in mother}   The map $\ms F\!\!\!\!\ms F\rightarrow s\ms D_Y(X)$ defined by sending $P$ to $\mu _P$ induces 
 a fully faithful functor of groupoids 
 $$\Xi:\ms F\!\!\!\!\ms F\to \Mom_Y(X)^\circ.$$
\end{cor}
\begin{proof}
Indeed this follows from \ref{L:combine} and the observation that  the functor $\Xi$ is fully faithful by the definition of the stack $\Mom_Y$. 
\end{proof}

\begin{remark} Note that in the above we consider only complexes on $X\times Y$ and morphisms between them, and not the resulting functors between derived categories.  Related to this we mention the observation of C\u{a}ld\u{a}raru that the map $\text{Hom}_{D(X\times Y)}(P, P')\rightarrow \text{Hom}(\Phi _P, \Phi _{P'})$ is not in general an isomorphism.
\end{remark}

\section{A Torelli theorem in the key of $\D$}

Fix \text{\rm K3}  surfaces $X$ and $Y$ over an algebraically closed field $k$.

In this section we prove the following derived category version of the Torelli theorem that has no characteristic restrictions.  It is similar to the classical Torelli theorem in that it specifies that some kind of ``lattice'' isomorphism preserves a filtration on an associated linear object.

\begin{thm}\label{T:derived torelli}
 If there is a kernel $P\in\D(Y\times X)$ inducing a filtered equivalence $\D(Y)\to\D(X)$ (see Paragraph \ref{Para:filt}) then $X$ and $Y$ are isomorphic.
\end{thm}

We prove some auxiliary results before attacking the proof.
By abuse of notation we will write $\Phi$ for both $\Phi_P$ and $\Phi_P^{A^*(-)_{\text{\rm num}, \Q}}$.  Let $C_X\subset \Pic(X)$ (resp.\ $C_Y\subset \Pic(Y)$) denote the ample cone of $X$ (resp.\ $Y$).

\begin{lem}\label{L:line em up}
 In the situation of Theorem \ref{T:derived torelli}, we may assume that $\Phi(1,0,0)=(1,0,0)$ and that the induced isometry $\kappa :\Pic(Y)\to\Pic(X)$ 
 has one of the following two properties:
\begin{enumerate}
\item[I.] $\kappa $ sends  $C_Y$ isomorphically to $C_X$.
\item[II.] $\kappa $ sends $C_Y$ isomorphically to $-C_X$.
\end{enumerate}
\end{lem}
\begin{proof}
By the assumption that $\Phi$ is filtered, we have that $\Phi(0,0,1) = (0,0,1)$ or $(0,0,-1)$. Composing with the shift functor $\ms F\mapsto\ms F[1]$ if necessary, we may assume that $\Phi(0,0,1)=(0,0,1)$.
 Since $\Phi$ is an isometry, $(1,0,0)\cdot(0,0,1)=-1$, and $(1,0,0)^2=0$, we see that there is some $b\in\Pic(X)$ such that
 $$\Phi(1,0,0)=\left(1,b,\frac{1}{2}b^2\right).$$
 Composing $\Phi$ by the twist with $-b$ yields a new Fourier-Mukai transformation sending $(0,0,1)$ to $(0,0,1)$ and $(1,0,0)$ to $(1,0,0)$ (by the formula \ref{E:Lformula}).  Since $\Pic(X)=(0,0,1)^{\perp}\cap(1,0,0)^{\perp}$ and similarly for $\Pic(Y)$, we see that such a $\Phi$ induces an isometry $\Pic(Y)\simto\Pic(X)$.  Following \cite[p. 366]{Ogus} set
$$
V_X:= \{x\in \Pic(X)_{\mathbb{R}}| x^2>0, \ \ \langle x, \delta \rangle \neq 0 \text{ \ for all $\delta \in \Pic (X)$ with $\delta ^2 = -2$}\}
$$
and define $V_Y$ similarly.
 Then since $\Phi $ is an isometry it induces an isomorphism $V_Y\rightarrow V_X$.
 By results of Ogus \cite[Proposition 1.10 and Remark 1.10.9]{Ogus}, we know that the ample cone is a connected component $V_X$, and that the group $R$ generated by reflections in $(-2)$-curves and multiplication by $-1$ acts simply transitively on the set of connected components of $V_X$, and similarly for $V_Y$.  In particular, there is some element $\rho:\Pic(X)\to\Pic(X)$ of this group $R$ such that the composition $\rho\circ\Phi:\Pic(Y)\simto\Pic(X)$ induces an isomorphism of ample cones.

 We claim that there is a representation of $R$ as a group of Fourier-Mukai autoequivalences of $X$ whose induced action on $A^*(X)_{\text{\rm num}}=\Z\oplus\Pic(X)\oplus\Z$ is  equal to the standard action on $\Pic (X)$ and multipliciation by $\pm 1$ on the outer summands (with the same sign for both factors). This will clearly complete the proof of the Lemma.

To define this embedding of $R$, suppose $C\subset X$ is a $(-2)$-curve.  The structure sheaf $\ms O_C$ is a spherical object of $\D(X)$ (see Definition \ref{D:spherical}), and the spherical twist $T_{\ms O_C}:\D(X)\simto\D(X)$ acts on $\H(X)$ by reflecting in the Mukai vector $(0,C,1)$.  Composing this twist with the tensoring equivalence $\tensor\ms O(C):\D(X)\to\D(X)$ gives a Fourier-Mukai equivalence whose induced action on $A^*(X)_{\text{\rm num}}=\Z\oplus\Pic(X)\oplus\Z$ is the identity on the outer summands and the reflection in $C$ on $\Pic(X)$.  Similarly, the shift isomorphism $\ms F\mapsto\ms F[1]:\D(X)\to\D(X)$ acts by $-1$ on $A^*(X)_{\text{\rm num}}$.  This establishes the claim.
\end{proof}

We will assume that our kernel $P$ satisfies the conclusions of Lemma \ref{L:line em up}.

\begin{prop}\label{P:space isom}
There is an isomorphism of infinitesimal deformation functors $\delta:\Def_X\to\Def_Y$ such that
 \begin{enumerate}
 \item $\delta^{-1}(\Def_{(Y,L)})=\Def_{(X,\Phi(L))}$ for all $L\in\Pic(Y)$;
 \item for each augmented Artinian $W$-algebra $W\to A$ and each $$(X_A\to A)\in\Def_{(X,H_X)}(A),$$ there is an object $$P_A\in\D(X_A\times_A \delta(X_A))$$ reducing to $P$ on $X\times Y$.
 \end{enumerate}
\end{prop}
\begin{proof}
Given an augmented Artinian $W$-algebra $W\to A$ and a deformation $X_A\to A$, let $\ms D_A$ denote the stack of unobstructed universally gluable relatively perfect complexes with Mukai vector $(0,0,1)$ on $X_A$.  By Corollary  \ref{P:open in mother}, the kernel $P$ defines an open immersion $Y\inj D_k$ such that the fiber product $\G_m$-gerbe $$\ms Y:=Y\times_{D_k}\ms D_k\to Y$$ is trivial (the complex $P$ defines a section of this gerbe).

Since $Y\to D_A\tensor k$ is an open immersion and $D_A$ is smooth over $A$, we see that the open subscheme $Y_A$ of $D_A$ supported on $Y$ gives a flat deformation of $Y$ over $A$, carrying a $\G_m$-gerbe $\ms Y_A\to Y_A$. Write $\ms P_A$ for the perfect complex of $\ms Y_A\times X_A$-twisted sheaves corresponding to the natural inclusion $\ms Y_A\to\ms D_A$. Write $\pi:\ms Y_A\times X_A\to Y_A\times X_A$.

\begin{prop}\label{P:untwist}
 With the  preceding notation, there is an invertible sheaf $\ms L$ on $\ms Y_A\times X_A$ such that the complex $$P_A:=\R(\pi_\ast\ms P_A\tensor\ms L_A^{\vee})\in\D(Y_A\times_A X_A)$$
 satisfies $$\L\iota^\ast P_A\cong P\in\D(Y\times X),$$
 where $\iota:Y\times X\inj Y_A\times_A X_A$ is the natural inclusion.
\end{prop}
\begin{proof}
 Consider the complex $\ms Q:=\ms P_A^{\vee}\tensor\pr_2^\ast\omega_{X_A}[2]$.  Pulling back by the morphism $$g:Y\times X\to \ms Y_A\times_A X_A$$ corresponding to $P$ yields the equality $$\L g^\ast\ms Q=P^\vee .$$  It follows that $$\R(\pr_1)_\ast(\ms Q)$$ is a perfect complex on $\mls Y_A$  whose pullback via the section $Y\to\ms Y$ is $\Phi^{-1}(\ms O_X)$.  Since $$\Phi(1,0,1)=(1,0,1),$$ this complex has rank $1$ and
$$\L g^\ast\det\R(\pr_1)_\ast(\ms Q)=\det\R(\pr_1)_\ast( P^\vee )\cong\ms O_Y.$$
 It follows that $$P\cong\ms P\tensor\det\R(\pr_1)_\ast(\ms Q).$$
Setting $$\ms L=\det\R(\pr_1)_\ast(\ms Q)^\vee$$
completes the proof.
\end{proof}

We can now prove part (ii) of Proposition \ref{P:space isom}:
the scheme $Y_A$ defined before Proposition \ref{P:untwist} gives a point of $\Def_Y(A)$, giving the functor $$\delta:\Def_X\to\Def_Y,$$ and Proposition \ref{P:untwist} shows that $P$ lifts to $$P_A\in\D(Y_A\times_A X_A),$$ as desired.

A symmetric construction starting with the inverse kernel $P^\vee$ yields a map $$\delta':\Def_Y\to\Def_X$$ and lifts $$P^\vee_A\in\D(\delta'(Y_A)\times Y_A).$$  Composing the two yields an endomorphism $$\eta:\Def_X\to\Def_X$$ and, for each $A$-valued point of $\Def_X$, lifts of $$P^\vee \circ P$$ to a complex $$Q_A\in\D(X_A\times\eta(X_A)).$$  But the adjunction map yields a quasi-isomorphism $$\ms O_{\Delta_X}\simto P^\vee \circ P,$$ so $Q_A$ is a complex that reduces to the sheaf $\ms O_{\Delta_X}$ via the identification $$\eta(X_A)\tensor k\simto X.$$  It follows that $Q_A$ is the graph of an isomorphism $$X_A\simto\eta(X_A),$$ showing that $\delta'\circ\delta$ is an automorphism of $\Def_X$, whence $\delta$ is an isomorphism.

Now suppose $Y_A$ lies in $\Def_{(Y,L)}$.  Applying $P_A^\vee$ yields a complex $C_A$ on $X_A$ whose determinant restricts to $\Phi(L)$ on $X$.  It follows that $X_A$ lies in $\Def_{(X,\Phi(L))}$, as desired.

This completes the proof of Proposition \ref{P:space isom}.
\end{proof}

\begin{proof}[Proof of Theorem \ref{T:derived torelli}]
Choose ample invertible sheaves $H_X$ and $H_Y$ that are not divisible by $p$ such that either  $H_X=\Phi(H_Y)$ (case I) or $H_X = -\Phi (H_Y)$ (case II).  Deligne showed in \cite{D} that there is a projective lift $(X_V,H_{X_V})$ of $(X,H_X)$ over a finite extension $V$ of the Witt vectors $W(k)$.   For every $n\geq $ let $V_n$ denote the quotient of $V$ by the $n$-th power of the maximal ideal, and let $K$ denote the field of fractions of $V$.

By Proposition \ref{P:space isom}, for each $n$ there is a polarized lift $(Y_n,H_{Y_n})$ of $(Y,H_Y)$ over $V_n$ and a complex $$P_n\in\D(Y_n\times X_n)$$
lifting $P$.  By the Grothendieck Existence Theorem, the polarized formal scheme $(Y_n,H_{Y_n})$ is algebraizable, so that there is a lift $(Y_V,H_{Y_V})$ whose formal completion is $(Y_n,H_{Y_n})$.

By the Grothendieck Existence Theorem for perfect complexes \cite[Proposition 3.6.1]{L}, the system $(P_n)$ of complexes is the formal completion of a perfect complex $$P_V\in\D(Y_V\times_V X_V).$$ In particular, $P_V$ lifts $P$ and Nakayama's Lemma shows that the adjunction maps $$\Delta_\ast\ms O_{X}\to P_V\circ P_V^\vee $$ and $$P_V^\vee \circ P_V\to\Delta_\ast\ms O_{Y}$$ are quasi-isomorphisms.  It follows that for any field extension $K'/K$, the generic fiber complex
$$P_{K'}\in\D(Y_{K'}\times_{K'} X_{K'})$$
induces a Fourier-Mukai equivalence
$$\Phi:\D(Y_{K'})\to\D(X_{K'}),$$
and compatibility of $\Phi$ with reduction to $k$ shows that $\Phi(0,0,1)=(0,0,1)$.  Choosing an embedding $K\inj\C$ yields a filtered Fourier-Mukai equivalence $$\D(Y_V\tensor\C)\simto\D(X_V\tensor\C).$$
Since $\Phi$ is filtered and induces an isometry of integral Mukai lattices, $\Phi$ induces a Hodge isometry $$\H^2(Y_V\tensor\C,\Z)\simto\H^2(X_V\tensor\C,\Z)$$ (see e.g.\ part (i) of the proof  \cite[Proposition 10.10]{H}), so that $Y_V\tensor\C$ and $X_V\tensor\C$ are isomorphic.  Spreading out, we find a finite extension $V'\supset V$ and isomorphisms of the generic fibers $X_{K'}\simto Y_{K'}$.  
By the following lemma it follows that $X$ and $Y$ are isomorphic, as desired.
\end{proof}

\begin{lem}\label{L:unique spec}
Let $Y$ and $Z$ be relative \text{\rm K3}  surfaces over a discrete valuation ring $R$.  If the generic fibers of $Y$ and $Z$ are isomorphic then so are the special fibers.
\end{lem}
\begin{proof}
Applying \cite[Theorem 1]{MM}, any isomorphism of generic fibers yields a birational isomorphism of the special fibers.  Since \text{\rm K3}  surfaces are minimal, this implies that the special fibers are in fact isomorphic, as desired.
\end{proof}
Note that we are not asserting that isomorphisms extend, only that isomorphy extends!

\section{Lifting kernels using the Mukai isocrystals}
\label{S:crystal lift}

Let $k$ be a perfect field of characteristic $p>0$, let $W$ be the ring of Witt vectors of $k$, and let $K$ be the field of fractions of $W$.

Fix \text{\rm K3}  surfaces $X$ and $Y$ over  $k$ with lifts $X_W$ and $Y_W$ over $W$.  The Hodge filtrations on the de Rham cohomology of $X_W/W$ and $Y_W/W$ give subspaces $\Fil^2_X\subset\H^2(X/K)\subset\widetilde H(X/K)$ and $\Fil^2_Y\subset\H^2(Y/K)\subset\widetilde H(Y/K)$, where $\widetilde H(X/K)$ and $\widetilde H(Y/K)$ denote the crystalline realizations of the Mukai motives.

\begin{thm}\label{T:kernel lift}
 Suppose $P\in\D(X\times Y)$ is a kernel whose associated functor $\Phi:\D(X)\to\D(Y)$ is fully faithful.  If $$\Phi^{\widetilde H}:\widetilde H(X/K)\to\widetilde H(Y/K)$$ sends $\Fil^2_X$ to $\Fil^2_Y$ then $P$ lifts to a perfect complex $P_W\in\D(X_W\times_W Y_W)$.
\end{thm}
\begin{proof}
We claim that it suffices to prove the result under the assumption that $\Phi(0,0,1)=(0,0,1)$.  Indeed, fix a $W$-ample divisor $\beta$ on $Y_W$.  Suppose
$$\Phi(0,0,1)=(r,\ell,s)$$
with $r>0$.  Since $\Phi$ preserves the Hodge filtration we see that $\ell\in\Fil^1_Y\H^2(Y/K)$, whence $\ell$ is unobstructed on $Y$  \cite[1.12]{Ogus2}.  Similarly, $$\Phi(1,0,0)=(r',\ell',s')$$ such that
$$\ell\cdot\ell'-rs'-r's=1,$$
and $\ell'$ must also lie in $\Fil^1_Y\H^2_{\text{cris}}(Y/K)$, so that $\ell'$ lifts over $Y_W$. Thus, the moduli space $\M_Y(r,\ell,s)$ lifts to a relative moduli space $M_{Y_W}(r,\ell,s)$, and there is a tautological sheaf $\ms E_W$ on $M_{Y_W}(r,\ell,s)\times_W Y$ defining a relative FM equivalence.  This induces an isometry of $F$-isocrystals $$\Phi_{\ms E}:\widetilde H(M_Y(r,\ell,s)/K)\simto\widetilde H(Y/K)$$ that sends $(0,0,1)$ to $(r,\ell,s)$.  The composition yields a FM equivalence
$$\Phi_Q:\D(X)\to\D(M_Y(v))$$
sending $(0,0,1)$ to $(0,0,1)$ and preserving the Hodge filtrations on Mukai isocrystals.  In addition, since $\ms E$ lifts to $\ms E_W$, we see that $P$ lifts if and only if $Q$ lifts.  Thus, we may assume that $\Phi(0,0,1)=(0,0,1)$, as claimed.

Since $\Phi$ is an isometry, it follows that $$\Phi^{-1}(1,0,0)=\left(1,b,\frac{1}{2}b^2\right)$$
for some $b\in\Pic(X)$. By Corollary  \ref{P:open in mother}, the kernel $P$ corresponds to a morphism
 $$\mu_P:X\to\Mom_Y$$
whose image in $\sMom_Y$ is an open immersion.  More concretely, if $\ms P$ denotes the universal complex on $\Mom_Y\times Y$, we have that
$$P=\L(\mu_P\times\id)^\ast\ms P.$$
Write $$\ms M=X\times_{\sMom_Y}\Mom_Y,$$ so that $\mu_P$ defines morphisms
$$X\to\ms M\to X$$ making $X$ a $\G_m$-torsor over $\ms M$.  The associated invertible sheaf is $\ms M$-twisted.

Since $\Mom_Y$ is smooth over $W$, there is a canonical formal lift $\mf X$ of $X$ over $W$, with a corresponding formal gerbe $\ms G\to\mf X$ lifting $\ms M$ such that there is a perfect complex of coherent twisted sheaves $\mf P\in\D(\ms G\times \widehat Y_W)$ lifting $\ms P|_{\ms M}$.  (Indeed, $\mf X$ is just the open subspace of the formal completion $\widehat\sMom_Y$ supported on $\mu(X)$.)

The complex $\R(\pr_1)_\ast\ms P$ is an invertible $\ms G$-twisted sheaf, defining an equivalence
$$\D(\mf X)\simto\D^{\twist}(\ms G).$$
Let $$\mf Q\in\D(\mf X\times\widehat Y_W)$$ be the kernel giving the composition
$$\D(\mf X)\simto\D^{\twist}(\ms G)\simto\D(\widehat{Y}_W).$$
Since the class of $\R(\pr_1)_\ast\ms P$ might differ from the twisted invertible sheaf associated to $X\to\ms M$, we have that the restriction
$$Q\in\D(X\times Y)$$
differs from $P$ by tensoring with an invertible sheaf $\ms L$ pulled back from $X$.  One can check that $\Phi_Q(1,0,1)=(1,0,1)$.  Since $b$ is the unique invertible sheaf $\ms L$ on $X$ such that tensoring with $\ms L$ sends
$$\left(1,b,\frac{1}{2}b^2\right)\text{ to }(1,0,0),$$ we see that $$Q\cong P\tensor\pr_1^\ast \ms O_X(b).$$

  We have that $v(b)=\Phi_P^{-1}(v(\ms O_Y))$ and $\Phi$ respects the Hodge filtrations on the Mukai isocrystals; since $\ms O_Y$ is unobstructed on $Y_W$, we therefore have that $$b\in\Fil^1_X\H^2_{\textrm{cris}}(X/K),$$
whence $b$ is unobstructed on $X_W$.
The complex $$\widehat P_W:=\mf Q\tensor\pr_1^\ast\ms O_{\mf X}(-b)\in\D(\mf X\times\widehat Y_W)$$ gives a formal lift of $P$.

Finally, by construction the isotropic subspace $$F\subset\H^2(X/K)$$ parametrizing the formal lift $\mf X$ is $\Phi^{-1}(\Fil^2\H^2_Y/K))$.  Since
$$\Phi^{-1}(\Fil^2\H^2_Y/K))=\Fil_X^2\H^2(X/K),$$
we conclude that $$\mf X=\widehat X_W.$$  Applying the Grothendieck Existence Theorem for perfect complexes as in \cite{L}, we get the desired lift $P_W\in\D(X_W\times_W Y_W)$.
\end{proof}

\begin{remark}
 If we had an integral version of the Mukai isocrystal and an integral version of our results then we could produce the lift $X_W$ from $Y_W$ via $\Mom_{Y_W}$.  Unfortunately, the Tate twist involved in the formation of $\widetilde H(Y/K)$ precludes a na\"ive extension to integral coefficients.
\end{remark}

\begin{remark}
Taking the cycle $Z:=\chern(P)\sqrt{\Td_{X\times Y}}$ giving the action on cohomology, we can see that Theorem \ref{T:kernel lift} gives a special case of the variational crystalline Hodge conjecture (see e.g.\ \cite[Conjecture 9.2]{MP}): the fact that $\Phi_P$ preserves the Hodge filtrations on the Mukai isocrystals means that $$[Z]\in\Fil^2\H^4(X\times Y/K).$$
Lifting the kernel $P$ to $P_W$ lifts the cycle, confirming the conjecture in this case. This could be interpreted as a kind of (weak) ``variational crystalline version'' of Mukai's original results on the Mukai-Hodge structure \cite{M}.
\end{remark}

\section{Every FM partner is a moduli space of sheaves}
\label{S:lift}

In this section we prove statement (1) in Theorem \ref{T:mainthm}

Fix \text{\rm K3}  surfaces $X$ and $Y$ over an algebraically closed field of characteristic exponent $p$.  Suppose $P$ is the kernel of a Fourier-Mukai equivalence $\D(Y)\to\D(X)$.  We now show that $Y$ is isomorphic to a moduli space of sheaves on $X$.

  Let $v=(r,L_X,s)=\Phi(0,0,1)$ be the Mukai vector of a fiber $P_y$ (hence all fibers).

\begin{lem}\label{L:change-rank}
 We may assume that $r$ is positive and prime to $p$ and that $L_X$ is very ample.
\end{lem}
\begin{proof}
First, if either $r$ or $s$ is not divisible by $p$ then we get the lemma by possibly composing with a shift and the spherical twist corresponding to $\mls O_X$, which up to sign interchanges $r$ and $s$ (see \cite[Example 10.9 (ii)]{H}).   So we will assume that both $r$ and $s$ are divisible by $p$ and show that we can compose with an autoequivalence of $X$ to ensure that $r$ is not divisible by $p$.

 Since $\Phi$ induces an isometry of numerical Chow groups, we have that there is some other Mukai vector $(r',\ell,s')$ such that $$(r,L_X,s)(r',\ell,s')=\ell\cdot L_X-rs'-r's=1.$$  Thus, since both $r$ and $s$ are divisible by $p$ we have that $\ell\cdot L_X$ is prime to $p$.  Consider the equivalence $D(X)\to D(X)$ given by tensoring with $\ell^{\tensor n}$ for an integer $n$.  This sends the Mukai vector
 $(r,L_X,s)$ to
 $$(r,L_X+rn\ell,s+n\ell\cdot L_X+\frac{n^2}{2}\ell^2r).$$
 It is elementary that for some $n$ the last component will be non-zero modulo $p$.  After composing with the spherical twist associated to $\mls O_X$ and shifting the complex we can swap the first and last components and thus find that $r$ is not divisible by $p$, as desired.

 Changing the sign of $r$ is accomplished by composing with a shift.  Making $L_X$ very ample is accomplished by composing with an appropriate twist functor.
\end{proof}

As discussed in section \ref{modulidef}, we can consider the moduli space $M_X(v)$  of sheaves on $X$ with Mukai vector $v$ (with respect to the polarization $L_X$ of $X$), and this is again a \text{\rm K3}  surface which is a FM partner of $X$.

\begin{prop}\label{P:everyone is moduli}
 With the notation from the beginning of this section, there is an isomorphism $Y\simto M_X(v)$.
\end{prop}
\begin{proof}
 Consider the composition of the equivalences $\D(Y)\to\D(X)\to\D(M_X(v))$ induced by the original equivalence and the equivalence defined by the universal bundle on $X\times M_X(v)$.   By assumption we have that $\D(Y)\to\D(M_X(v))$ sends $(0,0,1)$ to $(0,0,1)$, so it is filtered.  Theorem \ref{T:derived torelli} implies that $Y\cong M_X(v)$, as desired.
\end{proof}

This completes the proof of statement (1) in Theorem \ref{T:mainthm}.  For later use we conclude this section with some observations about the choice of $v$ for which $Y\simto M_X(v)$, which also makes use of the fact that $r$ in the above is prime to $p$.  Suppose that the equivalence $\Phi :D(Y)\rightarrow D(X)$ is given by $P$ on $Y\times X$ for which the Mukai vector $v = (r, L_X, s)$ of a fiber $P_y$ has the additional property that $r$ is positive and prime to $p$, and that $L_X$ is ample (as can be arranged by \ref{L:change-rank}).  If $M$ is any line bundle, then the composition of $\Phi $ with $\otimes M:D(X)\rightarrow D(X)$ is given by a complex on $Y\times X$ with Mukai vector $w$ given by (see \ref{E:Lformula})
\begin{equation}\label{E:changeit}
w:= (r, L_X+rc_1(M), s+L_X\cdot c_1(M)+rc_1(M)^2/2).
\end{equation}
From the proof of \ref{P:everyone is moduli} it follows that $Y\simeq M_X(w)$ for any such $w$ for which $L_X+rc_1(M)$ is  ample.

\begin{lem}\label{L:weirdness}
For any subgroup $\Gamma\subset\NS(X)$ of non-maximal rank, there exists a primitive Mukai vector $v=(r, \ell, s)$ such that the following hold:
\begin{enumerate}
\item $Y\simeq \M_X(v)$.
\item The map
$v\cdot(\ast):A^*(X)_{\text{\rm num}}\to\Z $
is surjective;
\item $\ell $ is an ample class.
\item  $\ell\not\in p\NS(X)+\Gamma$.
\end{enumerate}
\end{lem}
\begin{proof}
Start with any primitive Mukai vector $v = (r, \ell, s)$ so that (1) and (2) hold.  Now observe that conditions (1) and (2) are preserved if we replace $v$ by $w$ as in \ref{E:changeit}.

To prove the lemma, it therefore suffices to show that we can add a suitable ample class to $\ell $ to ensure conditions (3) and (4).  This is immediate since every element of
$$
\NS (X)/(p\NS (X)+\Gamma )
$$
can be represented by an ample class.
\end{proof}

\section{Finiteness results}
\label{S:finiteness}

Fix a \text{\rm K3}  surface $X$ over the algebraically closed field $k$.    In this section we prove statements (2) and (3) in \ref{T:mainthm}.

First consider the case when the Picard number of $X$ is $\leq 4$, in which case we need to show that $X$ has finitely many FM-partners.

If $X$ has finite height then by \cite[4.2]{LM} there is a lift $X_W$ of $X$ over $W$ such that the restriction map $\Pic(X_W)\to\Pic(X)$ is an isomorphism.  Since every partner of $X$ has the form $M_X(v)$ for some Mukai vector $v$, we see that any partner of $X$ is the specialization of a partner of the geometric generic fiber.  But the generic fiber has characteristic $0$, whence it has only finitely many FM partners by the Lefschetz principle and the known result over $\C$ (see \cite{BrM}).  Since specializations of \text{\rm K3}  surfaces are unique by Lemma \ref{L:unique spec}, we see that $X$ has only finitely many partners.

If $X$ is supersingular \footnote{Granting the recent results showing that any  supersingular K3 surface is Shioda supersingular \cite{Ch, LM, P} (recall that we are assuming $p>2$), this case never occurs.} then there is a flat deformation $X_t$ of $X$ over $\spec k\[t\]$ such that the generic fiber has finite height and the restriction map $\Pic(X_t)\to\Pic(X)$ is an isomorphism.  Indeed, choosing generators $g_1,\ldots,g_n$ for $\Pic(X)$, each $g_i$ defines a Cartier divisor $G_i$ in $\Def_X$.  Moreover, the supersingular locus of $\Def_X$ has dimension $9$ by  \cite[Proposition 14]{Og99}.  Thus, a generic point of the intersection of the $G_i$ lies outside the supersingular locus, and we are done since we can dominate any local ring by $k\[t\]$.  The argument of the previous paragraph then implies the result for $X$.

Next we consider the case when 
 $X$ has Picard number $\geq 5$.  In this case, as explained at the beginning of \cite[Section 3]{LM}, $X$ is Shioda-supersingular and has Picard number $22$.  In this case we will prove that in fact $X$ has a unique FM partner (namely itself).

This we again do by lifting to characteristic $0$.  The key result is \cite[Corollary 2.7 (2)]{HLOY}, which implies  that if $K$ is an algebraically closed field of characteristic $0$ and if $Z/K$ is a \text{\rm K3}  surface with Picard number $\geq 3$ and squarefree discriminant, then any FM partner of $Z$ is isomorphic to $Z$.

We use this by showing that if $X/k$ is Shioda-supersingular and $Y$ is a FM partner of $X$, then we can lift the pair $(X, Y)$ to a FM pair $(\mc X, \mc Y)$ over the ring of Witt vectors $W(k)$ such that the Picard lattice of the geometric generic fiber of $\mc X$ has rank $\geq 3$ and square-free discriminant.  Then by the result of \cite{HLOY} just mentioned, we conclude that the geometric generic fibers of $\mc X$ and $\mc Y$ are isomorphic whence $X$ and $Y$ are isomorphic.

So fix such a pair $(X, Y)$, and let's construct the desired lifting $(\mc X, \mc Y)$.
 
Let $N$ be the Picard lattice of $X$.  By our assumption that $X$ is supersingular, $N$ has the following properties (see for example \cite[1.7]{Ogus}):
\begin{enumerate}
\item $N$ has rank $22$.
\item Let $N^\vee $ denote the dual of $N$, and let $N\hookrightarrow N^\vee $ be the inclusion defined by the nondegenerate pairing on $N$.  Then the quotient $N^\vee /N$ is annihilated by $p$ and has dimension as a $\mathbb{F}_p$-vector space $2\sigma _0$ for some integer $\sigma _0$ between $1$ and $10$ ($\sigma _0$ is the \emph{Artin invariant}).
\end{enumerate}

Let $F\subset N$ be a rank $2$ sublattice with $N/F$ torsion free.  Applying lemma \ref{L:weirdness} with $\Gamma = F$, we can assume that $Y = M_X(v)$ with $v = (r, \ell , s)$ for $\ell \in N$ with nonzero image in $N/(pN+F)$.
Let  $E$ be the saturation of $F+\Z\ell $ in $N$.  By construction, the map $F/pF\to F^\vee/pF^\vee=(F/pF)^\vee$ is an isomorphism.

 There is a natural diagram

$$\xymatrix{E\ar[r]\ar[d] & N\ar[d]\\
E^\vee\ar[d] & N^\vee\ar[l]\ar[d]\\
E^\vee/E & N^\vee/N\ar[l]}$$
in which all four arrows in the bottom square are surjective.   In particular, $E^\vee /E$ is an $\mathbb{F}_p$-vector space.  Also if $Q$ denotes the quotient $E/F$, then $Q$ has rank $1$ and we have an exact sequence
$$
0\rightarrow Q^\vee /(Q^\vee \cap E)\rightarrow E^\vee /E\rightarrow F^\vee /(\text{Im}(E\rightarrow F^\vee ))\rightarrow 0.
$$
Since the quotient $F^\vee/F$ is already $0$, this shows that $E^\vee /E$
 is isomorphic to $0$ or $\Z/p\Z$.  In particular, $E$ has rank $3$ and square-free discriminant.

As explained in the appendix (in particular \ref{P:A.1}),  there is a codimension at most $3$ formal closed subscheme of the universal deformation space $D:=\spf W\[t_1,\ldots,t_{20}\]$ of $X$ over which $E$ deforms.  The universal deformation is algebraizable (as $E$ contains an ample class) and a geometric generic fiber is a \text{\rm K3}  surface over an algebraically closed field of characteristic $0$ with Picard lattice isomorphic to $E$.  Let
$$\mc X\to\spec R$$
be a relative \text{\rm K3}  surface with special fiber $X$ and geometric generic Picard lattice $E$.  Write $$\ms M\to\spec R$$ for the stack of sheaves with Mukai vector $(r,\ell,s)$ stable with respect to a sufficiently general polarization.  We know that $\ms M$ is a $\m_r$-gerbe over a relative $K3$ surface
$$\mc M\to\spec R,$$ and
by assumption we have that the closed fiber of $\ms M$ is ismorphic to $Y\times\B\m_r$.  Since $r$ is relatively prime to $p$, we have that the Brauer class associated to the gerbe $\ms M\to\mc M$ is trivial.  In particular, there is an invertible $\ms M$-twisted  sheaf $\ms L$ on $\ms M$ (see \cite{L2} for basic results on twisted sheaves).

Now let $\ms V$ be the universal twisted sheaf on $\ms M\times_R\mc X$ and $V$ the tautological sheaf on $Y\times X$.  Write
$$\pi:\ms M\times\mc X\to\mc M\times\mc X$$ for the natural projection and let
$$\ms W:=\pi_\ast\left(\ms V\tensor\ms L^\vee\right).$$  There is an invertible sheaf $\ms N$ on $Y$ such that
$$\ms W|_{Y\times X}\cong V\tensor\pr_1^\ast\ms N,$$
the kernel of another equivalence between $\D(X)$ and $\D(Y)$.
It follows from the adjunction argument in the proof of Proposition \ref{P:space isom} that $\ms W$ also gives a Fourier-Mukai equivalence between the geometric generic fibers of $\mc M$ and $\mc X$ over $R$.  By \cite{HLOY}, we have that $\mc M_{\widebar\eta}$ and $\mc X_{\widebar\eta}$ are isomorphic.  By specialization (using Lemma \ref{L:unique spec}), we see that $Y\cong X$, as desired.

This completes the proof of \ref{T:mainthm}. \qed



\appendix

\section{Deformations of \text{\rm K3} surfaces with families of invertible sheaves}

In this appendix, we prove the following result.

\begin{prop}\label{P:A.1}
	\label{P:sticky-Pic}
	Let $X$ be a  \text{\rm K3}  surface over an algebraically closed field $k$ of characteristic $\neq 2$,  and let $E\subset\Pic(X)$ be a saturated subgroup containing an ample divisor.  Suppose one of the following conditions hold:
\begin{enumerate}
\item [(i)] $k$ has characteristic $0$.
\item [(ii)] $X$ has finite height.
\item [(iii)] $X$ is supersingular, the characteristic of $k$ is $\neq 2$, and the rank of $E$ is $\leq 10$.
\end{enumerate}
Then there exists
	\begin{itemize}
		\item a complete dvr $R$ with fraction field $K$ of characteristic $0$ and residue field $\kappa$;
		\item a relative \text{\rm K3}  surface $\mc X\to\spec R$;
		\item an inclusion $k\to\kappa$ and an isomorphism $X\tensor_k\kappa\simto\mc X_\kappa.$
		\item an isomorphism $E\to\Pic(\mc X)$ specializing to the given inclusion $E\subset \Pic(X)$ whose restriction $E\to\Pic(\mc X_{\widebar K})$ is an isomorphism for any choice of algebraic closure $K\subset\widebar K$; 
	\end{itemize}
In other words, there is a lift of $X$ to characteristic $0$ whose geometric generic fiber has Picard group exactly $E$.
\end{prop}

Cases (i) and (ii) are already shown in \cite[4.2]{LM}, so the remainder of the appendix is devoted to the supersingular case (iii).

Let $(L_1,\ldots,L_m)$ be a  basis for $E$ with $L_1$ ample (which we may assume without loss of generality). We can define a moduli space over $W$ that contains $(X,L_1,\ldots,L_m)$ as a $k$-point: 

\begin{defn}
	\label{D:marked-moduli}
	Let $\ms M\to\spec W$ be the stack whose objects over a $W$-scheme $S$ are tuples $(\mc X,\ms L_1,\ldots,\ms L_m)$, where $\pi:\mc X\to S$ is a relative \text{\rm K3}  surface and $\mls L_1,\ldots, \mls L_m$ are invertible sheaves such that $\ms L_1$ is $\pi$-ample.
\end{defn}

\begin{lem}
	\label{L:algebraicity-with-bundles}
	The stack $\ms M$ is an Artin stack locally of finite type over $W$.
\end{lem}
\begin{proof}[Sketch of proof]
	Let $\ms N$ be the stack of polarized \text{\rm K3} surfaces, whose objects over $S$ are pairs $(\mc X,\ms L_1)$ with $\ms L_1$ an ample invertible sheaf. We know that $\ms N$ has a smooth covering by open subschemes of Hilbert schemes, hence is an Artin stack locally of finite type over $W$.  Writing $\ms X\to\ms N$ for the universal relative \text{\rm K3}  surface, the forgetful morphism $\ms M\to\ms N$ is given by the $m-1$-fold fiber product $$\ms Pic_{\ms X/\ms N}\times_{\ms N}\cdots\times_{\ms N}\ms Pic_{\ms X/\ms N}.$$ Since the Picard stack of a relative \text{\rm K3}  surface is an Artin stack locally of finite type, we conclude the same for $\ms M$.
\end{proof}
The surface $X$ together with the generators $(L_1,\ldots,L_m)$ yield an object of $\ms M_k$.

\begin{lem}
	\label{L:old-time}
	To prove Proposition \ref{P:sticky-Pic}, it suffices to show the existence of a chain of objects $\xi_0,\ldots,\xi_N$ of $\ms M$ such that $\xi_0=(X,L_1,\ldots,L_m)$, each $\xi_i$ is a specialization of $\xi_{i+1}$ along some local ring, and $\xi_N$ has Picard group $E$.
\end{lem}
\begin{proof}
	Indeed, one then finds that $\xi_0$ is in the closure $\ms N$ of the residual gerbe of $\ms M$ supported at $\xi_N$. Choosing any smooth cover of $\ms N$, we find that a point over $\xi_N$ will specialize to a point over $\xi_0$. Since $\ms N$ is locally Noetherian, we can then dominate the local ring by a dvr with the given properties, as desired.
\end{proof}

In other words, it suffices to prove the result in several generizations.
\begin{remark}
	This method only works because the stack $\ms M$ is globally defined and has a nice local structure. A proof of Proposition \ref{P:sticky-Pic} using purely formal methods is significantly more complex than the global proof offered here.
\end{remark}

In light of cases (i) and (ii) of Proposition \ref{P:sticky-Pic} which are already known, to complete the proof of \ref{P:sticky-Pic} in the supersingular case it suffices to prove the following:

\begin{lem}
	\label{L:reduce-height}
	If $k$ has positive characteristic, the object $(X,L_1,\ldots,L_m),$ with $m\leq 10$,  is the specialization of an object $(X',L_1',\ldots,L_m')$ with $X'$ a \text{\rm K3}  surface of finite height defined over an extension of $k$.
\end{lem}
\begin{proof}
	Let $D_X=\spec k\[x_1,\ldots,x_{20}\]$ denote the formal universal deformation space of $X$ over $k$. As observed in \cite{D}, each invertible sheaf $L_i$ determines a divisor in $D_X$. In particular, the complete local ring of $\ms M\tensor k$ at $(X,L_1,\ldots,L_m)$ is determined by $m$ equations $f_1,\ldots,f_m$. Moreover, since $L_1$ is ample, the universal deformation $(\mc X,\mc L_1,\ldots,\mc L_m)$ of the tuple $(X,L_1,\ldots,L_m)$ over $$\Spf k\[x_1,\ldots,x_{20}\]/(f_1,\ldots,f_m)$$ is algebraizable. 

	We know by  \cite[Proposition 14(2)]{Og99} that the closed subscheme of deformations of $X$ that have infinite height has dimension $9$. Thus, since $E$ has rank at most $10$ we see that $T:=\Spec k\[x_1,\ldots,x_{20}\]/(f_1,\ldots,f_m)$ has dimension at least $10$ and cannot lie entirely in this closed subscheme. The generic point of $T$ parametrizes a tuple with $X'$ of finite height, as desired.
\end{proof}

This completes the proof of \ref{P:sticky-Pic}. \qed

\providecommand{\bysame}{\leavevmode\hbox
to3em{\hrulefill}\thinspace}

\end{document}